\documentclass[11pt,a4paper,reqno]{amsart} 
\usepackage{amsfonts,amssymb,amscd,graphicx, amsmath,enumerate,verbatim,calc,hyperref,subcaption,lscape,blkarray,bbm,appendix}
\usepackage{fancybox}
\usepackage{color} 
\usepackage{cite} 
\usepackage[margin=1.3in]{geometry} 
\usepackage{setspace}
\newtheorem{lemma}{Lemma}[section]
\newtheorem{thm*}{Theorem}
\newtheorem{cor*}[thm*]{Corollary}

\newtheorem{thm}[lemma]{Theorem}
\newtheorem{cor}[lemma]{Corollary} 
\theoremstyle{definition} 
\newtheorem{Def}[lemma]{Definition}
\newtheorem{Not}[lemma]{Notations}
\newtheorem{exam}[lemma]{Example}  \theoremstyle{remark} 
\newtheorem{rem}[lemma]{Remark} 
\newtheorem{rems}[lemma]{Remarks}  
\newcommand{\F}{\mathcal{F}}

\title[Subshifts of Finite Type]{On the Perron root and eigenvectors associated with a subshift of finite type}
\author[Haritha Cheriyath]{Haritha Cheriyath}
\address{Department of Mathematics\\
Indian Institute of Science Education and Research Bhopal\\
Bhopal Bypass Road, Bhauri \\
Bhopal 462 066, Madhya Pradesh\\
India}
\email{harithacheriyath@gmail.com,\ charitha@iiserb.ac.in}
\author[Nikita Agarwal]{Nikita Agarwal}
\address{Department of Mathematics\\
Indian Institute of Science Education and Research Bhopal\\
Bhopal Bypass Road, Bhauri \\
Bhopal 462 066, Madhya Pradesh\\
India}
\email{nagarwal@iiserb.ac.in}

\date{}

\begin{document} 
\maketitle
\begin{abstract}
	In this paper, we describe the relationship between the Perron root and eigenvectors of an irreducible subshift of finite type with the correlation between the forbidden words in the subshift. In particular, we derive an expression for the Perron eigenvectors of the associated adjacency matrix. As an application, we obtain the Perron eigenvectors for irreducible $(0,1)$ matrices which are adjacency matrices for directed graphs. Moreover, we derive an alternate definition of the Parry measure in ergodic theory on an irreducible subshift of finite type.  
\end{abstract}


\section{Introduction}

The Perron-Frobenius Theorem~\cite{Perron,Frobenius} is one of the most celebrated results in matrix theory with vast applications within mathematics and in other disciplines such as engineering, social sciences, and network theory, see~\cite{Brin,Keyfitz_book,NM,Parry64,Penner}. We will only state a part of the result which is crucial for the results in this paper.

\subsection*{The Perron-Frobenius Theorem}
A square non-negative matrix $A$ is said to be \textit{irreducible} if for each $i,j$, there exists $\ell=\ell(i,j)\ge 1$ such that the $ij^{th}$ entry of $A^\ell$ is positive. Let $p(i)$ denote the greatest common divisor of all $k$ such that the $ii^{th}$ entry of $A^{k}$ is positive. For an irreducible matrix $p(i)=P$ for all $i$, and $P$ is known as the \textit{period} of the matrix. 
The matrix $A$ is said to be \textit{primitive} if there exists $\ell\ge 1$ such that each entry of $A^\ell$ is positive. The period of $A$ equals one if $A$ is a primitive matrix.  
A non-negative matrix which is not irreducible is known as a \textit{reducible} matrix.  \\~\\
Let $A$ be a non-negative irreducible matrix with period $P\ge 1$ and spectral radius $\theta$. Then the Perron-Frobenius Theorem states the following:
\begin{itemize}
	\item The spectral radius $\theta$ is positive and an eigenvalue of $A$. There are exactly $P$ eigenvalues on the circle with radius $\theta$ centered at the origin and are given by $\theta$ times the $P^{th}$ roots of unity. The eigenvalue $\theta$ is known as the \textit{Perron root} or \textit{Perron value}. Consequently when $A$ is primitive, $\theta$ is the largest eigenvalue of $A$ in modulus, that is, all other eigenvalues of $A$ have modulus strictly less than $\theta$. 
	\item Each of the eigenvalue with modulus $\theta$ is simple. The left and right eigenspaces corresponding to the Perron root $\theta$ are one-dimensional and there exists a left/right eigenvector which has all its entries positive known as the \textit{left/right Perron eigenvector}. 
	\item If $A$ is primitive and $V$ and $U$ are normalized right and left Perron eigenvectors such that $U^TV=1$, then $\lim_{k\rightarrow\infty}A^k/\theta^k=VU^T$, which is the spectral projection onto the one-dimensional eigenspace for $\theta$.  
\end{itemize}

If $A$ is reducible, using simultaneous row/column permutations, it can be transformed into a block upper triangular matrix where each non-zero diagonal blocks is irreducible. The spectral radius $\theta\ge 0$ of $A$ is an eigenvalue of $A$ with a left and a right eigenvector having non-negative entries. The eigenvalue $\theta$ equals the maximum of the Perron roots of its diagonal blocks and it need not be simple. We call $\theta$ the Perron root of $A$ and all its associated non-negative eigenvectors as Perron eigenvectors.

Irreducible $(0,1)$ matrices turn out to be the adjacency matrices of strongly connected directed graphs with at most a single edge from one vertex to another. The Perron root and eigenvectors play a crucial role in understanding the connectivity of the underlying graph, see~\cite{Brualdi} and references therein. 

There is a never ending quest to develop algorithms to efficiently compute or estimate the Perron root and eigenvectors of a given matrix, see~\cite{Markham,Dembele,Brauer,DZ,Elsner,Lu,Bunse,HP,Noda}. In this paper, we will present a combinatorial method to compute the Perron root and eigenvectors. The method will be, more generally, applicable to subshifts of finite type and binary matrices will be a special case. Certain tools from combinatorics will be used to compute the Perron root and eigenvectors. These techniques, along with concepts from ergodic theory, specifically the local escape rate, will be used to compute the normalization factor of the eigenvectors. 

\section{Preliminaries}

\subsection{Subshift of finite type}\label{subsec:sft}
Subshifts of finite type are used to model a large class of dynamical systems. The problem of counting the number $f(n)$ of allowed words of length $n$ in a subshift has applications to comma-free codes, games, pattern matching, and several problems in probability theory, including finding the number of events which avoid appearance of a given set of events as sub-events. We refer to~\cite{Combinatorial} and~\cite{String} for an extensive account of several applications. Since $f(n)$ generally does not have a simple explicit formula, we study its generating function $F(z)$. The function $F(z)$ is rational and its special form helps to understand the asymptotic behavior of $f(n)$. The generating function $F(z)$ is described using the correlation between forbidden words, which is a polynomial function representing the overlapping of one word onto another. For subshifts of finite type which are irreducible, there is a unique measure of maximal entropy, known as the Parry measure introduced in~\cite{Parry64}. It is defined using an irreducible adjacency matrix which encodes the dynamics of the subshift. The Parry measure is obtained using the Perron-Frobenius theorem applied on this adjacency matrix. The logarithm of the Perron root is the topological entropy of the subshift and the Perron eigenvectors capture the connectivity between words, see~\cite{LM_book}. 

For $q\geq 2$, let $\Sigma=\{0,1,\dots,q-1\}$ be the set of \textit{symbols} and $\Sigma^{\mathbb{N}}$ be the set of all \textit{one-sided sequences} with symbols from $\Sigma$. A \textit{word} of length $n$ with symbols from $\Sigma$ is a finite tuple, denoted as $w_1w_2\dots w_n$, for some $w_1,\dots ,w_n\in \Sigma$. A \textit{subshift of finite type} $X\subset \Sigma^{\mathbb{N}}$ is a set of all sequences that do not contain words from a finite collection. Such words are called \textit{forbidden words}. An \textit{allowed word in $X$} is a word which appears as a subword in a sequence in $X$. In general, the set of symbols can be infinite, the sequences can be bi-infinite, and there can be infinitely many forbidden words, we do not consider any of these cases in this paper. The collection of forbidden words is said to be \textit{minimal} if all subwords of the forbidden words are allowed in sequences in $X$. Such a minimal collection is unique for a given subshift. Note that a minimal collection can have words of different lengths. If $\F$ is a minimal forbidden collection of finitely many words which describe $X$, we denote $X$ as $\Sigma_\F$. We assume $\F$ does not contain symbols, that is, words of length one (since if a symbol is forbidden, it can be removed from the symbol set $\Sigma$ altogether). A subshift of finite type $\Sigma_\F$ is said to be a \textit{one-step subshift} if all the words in $\F$ have length two. 

Every subshift of finite type is conjugate to a one-step shift via a block map (see~\cite{LM_book}). Consider a subshift of finite type $\Sigma_\F$. Let the longest word in $\F$ have length $p\ge 2$. Every sequence in $\Sigma_\F$ can be visualized as a sequence with allowed words of length $p-1$ as symbols which overlap progressively, that is 
\begin{eqnarray}\label{eq:conj}
	x_1x_2x_3\dots &\rightarrow& (x_1\dots x_{p-1})(x_2\dots x_p)(x_3\dots x_{p+1})\dots. 
\end{eqnarray}
This is known as a \textit{sliding block map}, where the sequence on the right has the following property: any two consecutive symbols $(x_1\dots x_{p-1})$ and $(y_1\dots y_{p-1})$ satisfy $x_2\dots x_{p-1}=y_1\dots y_{p-2}$ and $x_1\dots x_{p-1}y_{p-1}$ is an allowed word of length $p$ in $\Sigma_\F$. This defines a conjugacy between $\Sigma_\F$ and a one-step subshift with symbols from the collection consisting of all the allowed words of length $p-1$ in $\Sigma_\F$. 

The \emph{adjacency matrix} of this one-step shift is defined as follows: let $A$ be a binary matrix with rows and columns indexed by all the allowed words of length $p-1$ with symbols from $\Sigma$. To avoid confusion, without loss of generality, we order the words in \textit{lexicographic or dictionary} order\footnote{For any two distinct words $x=x_1\dots x_m,y=y_1\dots y_m$ of same length, the lexicographic order $\prec$ is defined as $x\prec y$ if there exists $1\le k\le m$ such that $x_i=y_i$ for all $i=1,\dots,k-1$ and $x_k<y_k$.}. The $(x_1\dots x_{p-1})(y_1\dots y_{p-1})^{th}$ entry of $A$ is 1 if and only if the words $x_1\dots x_{p-1}$ and $y_1\dots y_{p-1}$ overlap progressively, that is, $x_2\dots x_{p-1}=y_1\dots y_{p-2}$, and the word $x_1\dots x_{p-1}y_{p-1}$ is an allowed word of length $p$ in $\Sigma_\F$. The sum of the entries of $A^n$ gives the number of allowed words of length $n+p-1$ in $\Sigma_\F$.

The conjugacy~\eqref{eq:conj} is relevant for $p\ge 3$. If $p=2$, $\Sigma_\F$ is itself a one-step shift and thus is conjugate to itself and its adjacency matrix has size $q$.  We say that $\Sigma_\F$ is an \emph{irreducible (primitive, respectively)} subshift of finite type if and only if the one-step shift to which it is conjugate (as defined above) is irreducible (primitive, respectively), that is, its adjacency matrix $A$ is irreducible (primitive, respectively). A subshift of finite type which is not irreducible is known as \emph{reducible}. For convenience, we will call $A$ the adjacency matrix of $\Sigma_\F$ as well. We emphasize again that $p\ge 2$ since otherwise the forbidden symbols can be removed from the symbols set $\Sigma$ altogether.

Let us consider the following examples of subshifts of finite type $\Sigma_\F$ (where $\F$ is a minimal collection, by definition). For $q=2$ and $\F=\{000,11\}$ ($p=3$), $\Sigma_\F$ is conjugate to a one-step shift with the symbols set as the collection of all the allowed words of length two in $\Sigma_\F$, via the block map~\eqref{eq:conj}. The allowed words of length two in $\Sigma_\F$ are given by $\{00, 01, 10\}$ (in lexicographical order), and the adjacency matrix is given by 
\[
A = 
\begin{blockarray}{cccc}
	&00 & 01 & 10 \\
	\begin{block}{c(ccc)}
		00 & 0 & 1 & 0  \\
		01 & 0 & 0 & 1  \\
		10 & 1 & 1 & 0  \\
	\end{block}
\end{blockarray}\ .
\] 
It is easy to check that the matrix $A$ is primitive, hence the subshift of finite type $\Sigma_\F$ is primitive. 

\subsection{Tools from combinatorics}
In this section, we discuss some tools used from combinatorics. Let $\F=\{a_1,\dots,a_s\}$ be a \textit{reduced} collection of words with symbols from $\Sigma$, that is, for any $i\ne j$, $a_i$ is not a subword of $a_j$. Note that $\F$ can contain words of different lengths here. A crucial observation which will be used extensively in this paper is that a minimal collection of words is also reduced. For each natural number $n$, let $f(n)$ denote the number of words of length $n$ with symbols from $\Sigma$ that do not contain any word from the collection $\F$. By convention, $f(0)=1$. 

It is well-known that the space $\Sigma_\F$ is a compact metric space and the left shift map $\sigma:\Sigma_\F\to\Sigma_\F$ defined as $\sigma(x_1x_2x_3\dots)=x_2x_3\dots$ is continuous. Also, if $\Sigma_\F$ is irreducible, the \textit{topological entropy} $h_{top}(\Sigma_{\F})$ of the map $\sigma$ on $\Sigma_\F$ is given by $\lim_{n \to \infty}(\ln  f(n))/n$, which equals $\ln\theta$, where $\theta$ is the Perron root of the adjacency matrix of the subshift $\Sigma_\F$, see~\cite{LM_book}. Define the \textit{generating function} $F(z)$ for $(f(n))_n$ as $F(z)=\sum_{n=0}^\infty f(n)z^{-n}$. 

In~\cite{Combinatorial}, Guibas and Odlyzko introduced the notion of correlation between two words (strings) which quantifies their overlap. Moreover, they gave a formula for the generating function $F(z)$ through a system of linear equations involving correlation between forbidden words. Their work serves as the basis for results that we present in this paper. 

We first give the definition of the correlation polynomial between two words which plays a key role in this work.

\begin{Def} \label{def:corr}
	Let $x$ and $y$ be two words of lengths $p_1$ and $p_2$, respectively, with symbols from $\Sigma$. The \textit{correlation polynomial} of $x$ and $y$ is defined as
	\[
	(x,y)_z = \sum_{\ell=0}^{p_1-1} b_\ell z^{p_1-1-\ell}, 
	\]
	where $b_{\ell}=1$, if and only if the overlapping parts of $x$ and $y$ are identical when the left-most symbol of $y$ is placed right below the $(\ell+1)^{th}$ symbol of $x$ (from the left) and $b_\ell=0$, otherwise. The polynomial $(x,x)_z$ is said to be the \textit{auto-correlation polynomial} of $x$, and when $x\neq y$, the polynomial $(x,y)_z$ is said to be the \textit{cross-correlation polynomial} of $x$ and $y$. 
\end{Def}

\begin{exam}
	To understand the concept of the correlation polynomial, let us consider the following example. Let $x=101001$ ($p_1=6$), $y=10010$ ($p_2=5$). Then
	\begin{center}
		\begin{tabular}{c c c c c c c c c c c }
			$\ell$&&	\textbf{1} &\textbf{0} &\textbf{1} & \textbf{0} & \textbf{0} & \textbf{1}  &&& $b_{\ell}$ \\ 
			0&&   1&0&0&1&0& &&&0\\
			1&&	  & 1 &0 & 0 & 1 &0  &&&0\\ 
			2&&	  &&\textbf{1}&\textbf{0}&\textbf{0}&\textbf{1} &&&1\\
			3&&	  &&&1&0&0  &&&0\\
			4&&	  &&&&1&0  &&&0\\
			5&&	  &&&&&\textbf{1} &&&1
		\end{tabular}.
	\end{center}
	We get $(x,y)_z=z^3+1$. Similarly $(y,x)_z=z$, $(x,x)_z=z^5+1$, $(y,y)_z=z^4+z$.
\end{exam}

The following result gives an expression for the generating function $F(z)$ in terms of the correlation between the forbidden words.\\

\begin{thm}~\cite[Theorem 1]{Combinatorial}\label{thm:formF}
	Let $\F=\{a_1,\dots,a_s\}$ be a reduced collection of words with symbols from $\Sigma$. Let $F(z)$, $F_i(z)$ denote the generating functions for $f(n)$ and $f_i(n)$, respectively, where $f(n)$ denotes the number of words of length $n$ with symbols from $\Sigma$ not containing any of the words from $\F$, and $f_i(n)$ denotes the number of words of length $n$ with symbols from $\Sigma$ not containing any of the words from $\F$ except a single appearance of $a_i$ at the end. Then $F(z)$, $F_i(z)$ satisfy the linear system of equations 
	\begin{eqnarray*}
		K(z)\begin{pmatrix}F(z) \\ F_1(z) \\ \vdots \\ F_s(z)\end{pmatrix} = \begin{pmatrix}z \\ 0 \\ \vdots \\ 0\end{pmatrix},
	\end{eqnarray*}
	where $K(z)=\begin{pmatrix}z-q & z \mathbbm{1}^T \\ \mathbbm{1} &-z\mathcal{M}(z)\end{pmatrix}$, $\mathcal{M}(z)=((a_j,a_i)_z)_{1\le i,j\le s}$ is the correlation matrix for the collection $\F$, $\mathbbm{1}$ denotes the column vector of size $s$ with all 1's. Hence, 
	\begin{eqnarray}\label{eq:solF}
		\begin{pmatrix}F(z) \\ F_1(z) \\ \vdots \\ F_s(z)\end{pmatrix} = K(z)^{-1}\begin{pmatrix}z \\ 0 \\ \vdots \\ 0\end{pmatrix} = \dfrac{1}{(z-q)+r(z)}\begin{pmatrix}z\\ \mathcal{M}(z)^{-1}\mathbbm{1}\end{pmatrix},
	\end{eqnarray}
	where $q$ is the size of $\Sigma$ and $r(z)$ is the sum of the entries of $\mathcal{M}(z)^{-1}$.
\end{thm}

\noindent The following result is a straightforward consequence of Theorem~\ref{thm:formF}.

\begin{cor}\label{thm:formG}
	Let $\F=\{a_1,\dots,a_s\}$ be a reduced collection of words with symbols from $\Sigma$. Let $G_i(z)$ denote the generating function for $g_i(n)$, where $g_i(n)$ denotes the number of words of length $n$ with symbols from $\Sigma$ not containing any of the words from $\F$ except a single appearance of $a_i$ at the beginning. Then $F(z)$, $G_i(z)$ satisfy the linear system of equations 
	\begin{eqnarray*}
		L(z)\begin{pmatrix}F(z) \\ G_1(z) \\ \vdots \\ G_s(z)\end{pmatrix} = \begin{pmatrix}z \\ 0 \\ \vdots \\ 0\end{pmatrix},
	\end{eqnarray*}
	where $L(z)=\begin{pmatrix}z-q & z \mathbbm{1}^T \\ \mathbbm{1} &-z\mathcal{M}(z)^T\end{pmatrix}$. \\
	\noindent Consequently,
	\begin{eqnarray*}
		\begin{pmatrix}F(z) \\ G_1(z) \\ \vdots \\ G_s(z)\end{pmatrix} = L(z)^{-1}\begin{pmatrix}z \\ 0 \\ \vdots \\ 0\end{pmatrix} = \dfrac{1}{(z-q)+r(z)}\begin{pmatrix}z\\ (\mathcal{M}(z)^T)^{-1}\mathbbm{1}\end{pmatrix}.
	\end{eqnarray*}
\end{cor}
\begin{proof}
	Replace each word $a_i$ by its reverse $\hat{a_i}$, also observe that $(\hat{a_i},\hat{a_j})_z=(a_j,a_i)_z$. Hence the result follows.  
\end{proof}

\section{Summary of main results} 
Let $\Sigma_\F$ be a subshift of finite type, where $\Sigma=\{0,1,\dots,q-1\}$ and $\mathcal{F}=\{a_1,\dots,a_s\}$ be a minimal collection of words with the longest word having length $p\ge 2$. Let $A$ be the adjacency matrix of $\Sigma_\F$, as defined in Section~\ref{subsec:sft}. Let $\theta$ be the Perron root of $A$. 

\begin{Not}
	In addition to the notations in place, we use these notations in the following statements: 
	\begin{itemize}
		\item $\mathcal{M}(z)=[(a_j,a_i)_z]_{1\le i,j\le s}$: the matrix function of correlation polynomials between the words in the collection $\F$ (see Definition~\ref{def:corr}).
		\item $r(z)$: the rational function given by the sum of the entries of $\mathcal{M}(z)^{-1}$. It is easy to check that $r(z)$ is a well-defined function since the determinant of the matrix function $\mathcal{M}(z)$ is a non-constant polynomial. 
		\item $\mathcal{R}_i(z)$ (resp. $\mathcal{C}_j(z)$): the rational function given by the sum of the entries of the $i^{th}$ row (resp. $j^{th}$ column) of $\mathcal{M}(z)^{-1}$.
		\item $\widetilde{a_i}$: the subword of $a_i$ obtained by removing the first symbol of $a_i$. 
	\end{itemize}
\end{Not}

\begin{thm*} (The Perron root of the adjacency matrix $A$)
	The Perron root $\theta$ is the largest positive real zero in modulus of the rational function $(z-q)+r(z)$. Moreover, there is no zero outside the closed disk centered at the origin with radius $\theta$. Further, if $\Sigma_\F$ is irreducible, then $\theta$ is a simple zero of the rational function $(z-q)+r(z)$.
\end{thm*}

\begin{thm*}
	(Left and right normalized eigenvectors of the adjacency matrix $A$ corresponding to the Perron root $\theta$) Let $\Sigma_\F$ be irreducible. Let $v=(v_x)_x$ and $u=(u_x)_x$ (indexed by allowed words of length $p-1$) be the vectors defined as follows. For an allowed word $x$ of length $p-1$, 
	\[
	u_x=1-\sum_{i=1}^s \mathcal{R}_i(\theta) (\widetilde{a_i},x)_\theta, \ v_x=1-\sum_{j=1}^s\mathcal{C}_j(\theta)(x,a_j)_\theta,
	\] 
	where $(x,y)_\theta$ is $(x,y)_z$ evaluated at $z=\theta$.
	For each word $x$, $u_xv_x>0$. Moreover, $v$ and $u$ are right and left Perron eigenvectors, respectively, of $A$.\\
	Further, if $\theta>1$, the normalization factor for these Perron eigenvectors is given by 
	\[u^Tv =\theta^{p-1}\left(1+r'(\theta)\right).
	\]
\end{thm*}

There are three main components of the results: description of the Perron root $\theta$ (Theorem~\ref{thm:same_roots}), expressions for the right and left eigenvectors (not necessarily normalized) $v$ and $u$ (Theorem~\ref{thm:eig_vec}), and an expression for the normalization factor $u^Tv$ (Theorem~\ref{thm:eig_prod}). Now we describe two applications of our main results. These applications will be proved in Section~\ref{sec:app}.

\subsection*{Application to graph theory}

Let $\mathcal{G}$ be a directed graph with $n$ vertices and having adjacency matrix $A$. We assume that there is at most one edge between any pair of vertices. Using the results discussed above, we obtain an expression for the Perron root and Perron eigenvectors of $A$. 

\begin{thm*}
	Let $A=[A_{xy}]_{1\le x,y\le n}$ be a $(0,1)$ irreducible matrix of size $n$ with Perron root $\theta$. Let $\F=\{xy \ \vert\ A_{xy}=0,\ 1\le x,y\le n\}$, labelled as $\{a_1,\dots, a_s\}$. Let $u=(u_x)_{1\le x\le n}$
	and $v=(v_x)_{1\le x\le n}$ be the vectors defined as 	
	\[
	u_x=1-\sum\limits_{\substack{i=1\\a_i\text{ ends with }x}}^s \mathcal{R}_i(\theta), \ \ v_x=1-\sum\limits_{\substack{i=1\\a_i\text{ begins with }x}}^s \mathcal{C}_i(\theta).
	\]
	Then for each $x$, $u_xv_x>0$. Moreover, $v$ and $u$ are right and left Perron eigenvectors, respectively, of $A$. \\
	Further, if $\theta>1$, the normalization factor for these Perron eigenvectors is given by 
	\[u^Tv =\theta(1+r'(\theta)).
	\]
\end{thm*}

\begin{rems}
	1) In the preceding result, since each $a_j$ has length two, the correlation matrix $\mathcal{M}(z)$ of size $s$ has linear polynomials $z+\alpha$ on the diagonal with $\alpha$ either 0 or 1, and either 0 or 1 on the off-diagonal. More precisely, $(a_j,a_j)_z=z+1$ if and only if $a_j=uu$, for some $1\le u\le n$, and for $j\ne k$, $(a_j,a_k)_z=1$ if and only if $a_j=uv$ and $a_k$ begins with $v$, for some $1\le u,v\le n$.  \\
	2) By the Perron-Frobenius theorem and the preceding result, for a primitive matrix with $\theta>1$, we obtain an estimate for the number of paths $(A^k)_{xy}$ of length $k$ from the vertex $x$ to $y$ in the graph $\mathcal{G}$. The estimate is given by
	\[
	(A^k)_{xy}\ \sim \ \theta^{k-1}\dfrac{\left(1-\sum\limits_{\substack{i=1\\a_i\text{ begins with }x}}^s \mathcal{C}_i(\theta)\right)\left(1-\sum\limits_{\substack{i=1\\a_i\text{ ends with }y}}^s \mathcal{R}_i(\theta)\right)}{1+r'(\theta)}.
	\]
\end{rems}

\subsection*{Application to ergodic theory}
The Perron-Frobenius theorem for irreducible matrices have played a crucial role in several areas of mathematics including ergodic theory. In~\cite{Parry64}, Parry showed the existence and uniqueness of a measure of maximal entropy for irreducible subshifts of finite type using the Perron-Frobenius Theorem. This measure is now called the \textit{Parry measure}, which we will now define. 

Let $\Sigma_\F$ be an irreducible subshift of finite type, where $\F$ is a non-empty minimal collection of words with the longest word having length $p\ge 2$. Let $w=w_1\dots w_n$ be an \emph{allowed word} in $\Sigma_{\F}$ (that is, it does not contain any word from the collection $\F$ as a subword) and let 
\[
C_w=\{x_1x_2\dots\in\Sigma_{\F}\ \vert\ x_1=w_1,x_2=w_2,\dots, x_n=w_n\},
\]
denote the \textit{cylinder based at the word $w$}. Then we obtain a probability measure space with set $\Sigma_{\F}$, $\sigma$-algebra generated by the cylinders based at all allowed words of finite length, and the measure $\mu$. The measure $\mu$ is the pull-back of the \textit{Parry measure} on the one-step shift via the conjugacy~\eqref{eq:conj}. 
The measure $\mu$ will be called the Parry measure on $\Sigma_\F$ as well. It is defined as follows: for every allowed word $w=w_1\dots w_n$ ($n\ge p$),
\begin{eqnarray}\label{eq:pmold}
	\mu(C_w)&=&\frac{U_{w_1\dots w_{p-1}}V_{w_{n-p+2}\dots w_n}}{\theta^{n-p+1}},
\end{eqnarray}
where $\theta>0$ is the Perron root of $A$, $V$ and $U$ are normalized right and left (column) Perron eigenvectors, respectively, such that $U^TV=1$. For any allowed word $w$ of length $n$ ($1\le n<p$), $\mu(C_w)$ can be computed using the fact that $C_w$ is a union of all the disjoint cylinders $C_{w'}$ with $w'$ an allowed word of length $p$ that starts with $w$. Note that $1\le\theta\le q$ since each row/column of $A$ has at most $q$ 1's and at least one 1 if $A$ is irreducible. The Parry measure $\mu$ has the following properties.

\begin{itemize}
	\item In the case of the full shift ($\F=\emptyset$), the Parry measure $\mu$ is the uniform (Bernoulli) probability measure on $\Sigma^\mathbb{N}$. The cylinder $C_w$ is the collection of all sequences beginning with the word $w$ of length $n\ge 1$ with symbols from $\Sigma$, with $\mu(C_w)=1/q^n$. 
	\item The left shift map $\sigma:\Sigma_\F\rightarrow \Sigma_{\F}$ is measure-preserving and ergodic with respect to $\mu$.
	\item If $\F\ne\emptyset$, it is immediate from~\eqref{eq:pmold} that two cylinders based at words of identical length need not have the same measure. But the measure of all the cylinders based at words of identical length $n$ with the same starting $(p-1)$-word and the same ending $(p-1)$-word is the same. The number of such words is determined by $A_{(w_1\dots w_{p-1})(w_{n-p+2}\dots w_n)}^{n-p+1}$, the $(w_1\dots w_{p-1})(w_{n-p+2}\dots w_n)^{th}$ entry of $A^{n-p+1}$.
\end{itemize}

An immediate consequence of the previously stated results is an alternate definition of the Parry measure. In the following result, $\Sigma_\F$ is irreducible with positive topological entropy, that is, the Perron root satisfies $\theta>1$.

\begin{thm*}
	Let $w$ be an allowed word of length $n\ge p$ in $\Sigma_\mathcal{F}$ which starts with a word $x$ of length $p-1$ and ends with a word y of length $p-1$. Then  
	\begin{eqnarray*}
		\mu(C_w)&=&\dfrac{\left(1-\sum\limits_{i=1}^s \mathcal{R}_i(\theta) (\widetilde{a_i},x)_\theta\right)\left(1-\sum\limits_{j=1}^s \mathcal{C}_j(\theta) (y,a_j)_\theta\right)}{\theta^{n} \left(1+r'(\theta)\right)},\hspace{.3in}
	\end{eqnarray*}
	where $C_w$ is the cylinder based at word $w$ in $\Sigma_\F$.
\end{thm*}

The term on the right in the above expression is well-defined, which will be proved in due course.

\section{The Perron root}
In this section, we prove that the Perron root $\theta$ of the subshift $\Sigma_\F$ is the largest positive real (simple) pole of the generating function $F(z)$. Consequently from Theorem~\ref{thm:formF}, $\theta$ is the largest zero of the rational function $z-q+r(z)$.

\begin{thm}\label{thm:same_roots}
	The (rational) generating function $F(z)$ is analytic outside the closed disk centered at the origin with radius $\theta$, the Perron root of the adjacency matrix for the subshift $\Sigma_{\F}$. Moreover, $\theta$ is a pole of $F$. The pole $\theta$ is simple if $\Sigma_{\F}$ is irreducible. 
\end{thm}
\begin{proof}
	Let $\Sigma_{\F}$ be an irreducible subshift where $\F$ is a minimal collection with the longest word having length $p\ge 2$. Let $A$ be the adjacency matrix of $\Sigma_\F$.\\
	\textbf{Step 1}: The eigenvalues of $A$ with modulus $\theta$ are $\theta \omega,\dots,\theta \omega^P$, where $P\ge 1$ is the period of $A$ and $\omega\ne 1$ is a $P^{th}$ root of unity. By the (complex) Jordan decomposition of $A$, we get $f(n)=\left(c_1\omega^n+\dots+c_{P}\omega^{Pn}\right)\theta^n+e(n)$, where $e(n)=O(n^\alpha\lambda^n)$ for some $0<\lambda<\theta$ and integer $\alpha\ge 0$. For each $\ell\in\mathbb{N}$, let $n_\ell=P\ell$. Then $f(n_\ell)=C\theta^{n_\ell}+e(n_\ell)$, where $C=c_1+c_2+\dots+c_{P}$. The constant $C$ is non-zero since otherwise $f(n_\ell)=e(n_\ell)$ and hence  $\limsup_{n\rightarrow \infty} e(n)^{1/n}=\theta$ which is a contradiction. Therefore $\lim_{\ell\rightarrow \infty}f(n_\ell)\theta^{-n_\ell}\rightarrow C\ne 0$. Hence the series $\sum_{n=0}^\infty f(n)\theta^{-n}$ diverges. Hence $\theta$ is a pole of $F$. \\ 
	\textbf{Step 2}: Using Neumann series expansion, for $|z|>\theta$, 
	\[
	R(z)=(zI-A)^{-1}=\sum_{n=0}^\infty \frac{A^n}{z^{n+1}},
	\] 
	where $R(z)$ is the resolvent of $A$. Since $\sum_{i,j}A^n_{i,j}=f(n+p-1)$, we obtain, on $|z|>\theta$,
	\begin{equation}\label{eq:res}
		F(z)=\frac{\sum_{i,j}R_{i,j}(z)}{z^{p-2}}+f(0)+\frac{f(1)}{z}+\dots+\frac{f(p-2)}{z^{p-2}},
	\end{equation} where $A_{i,j}^n$ and $R_{i,j}$ denote the $ij^{\text{th}}$ entries of the matrices $A^n$ and $R$, respectively. By the identity theorem,~\eqref{eq:res} holds true for all $z\in\mathbb{C}$. 
	Also, $R(z)=\frac{\text{adj}(zI-A)}{\det(zI-A)}$ (note that both $\text{adj}(zI-A)$ and $\det(zI-A)$ are polynomials). Since $\theta$ is the largest root of $\det(zI-A)$, $\sum_{i,j}R_{i,j}(z)$ is analytic outside the closed disk centered at the origin with radius $\theta$. Consequently, from~\eqref{eq:res}, $F(z)$ is analytic outside the closed disk centered at the origin with radius $\theta$. \\
	The arguments in Steps 1 and 2 can be extended to a reducible subshift of finite type provided at least one diagonal block (in the block upper triangular matrix via simultaneous row/column permutations associated to the adjacency matrix) is irreducible. In that case $\theta\ge 1$.\\
	\textbf{Step 3}: Since $\Sigma_\F$ is irreducible, $\theta$ is a simple root of $\det(zI-A)$. Hence $(z-\theta)\sum_{i,j}R_{i,j}(z)$ is analytic at $\theta$. Consequently, using~\eqref{eq:res}, $(z-\theta)F(z)$ is analytic at $\theta$. Therefore $\theta$ is a simple pole of $F$.     
\end{proof}

\begin{exam}
	In graph theory, the star graph $S_{1,n-1}$ ($n\ge 2$) is an undirected graph with one vertex in the center and $n-1$ vertices connected to it; there is no other edge. The Perron root associated to $S_{1,n-1}$ is $\sqrt{n-1}$. We verify this using Theorem~\ref{thm:same_roots}. For each $n$, let $\Sigma_{\F_n}$ be the subshift with adjacency matrix equal to that of $S_{1,n-1}$. Then $\F_n=\{00\}\cup\{ab\ :\ 1\le a,b\le n-1\}$ for which the correlation matrix polynomial function is
	\[
	\mathcal{M}_n(z)=\begin{pmatrix}
		z+1&0\\0&K_n(z)
	\end{pmatrix},	
	\]   with
	\[
	K_n(z)=D_n(z)+\begin{pmatrix}
		J_{n,1}&J_{n,1}&\dots&J_{n,1}\\
		J_{n,2}&J_{n,2}&\dots&J_{n,2}\\
		\vdots&\vdots&\ddots&\vdots\\
		J_{n,n-1}&J_{n,n-1}&\dots&J_{n,n-1}
	\end{pmatrix},
	\]
	where $D_n(z)$ is the diagonal matrix of order $(n-1)^2$ with all the diagonal entries as $z$, and $J_{n,i}$ is a matrix of order $n-1$ whose entries on the $i^{\text{th}}$ column are all 1 and the other entries are all zero. Hence the sum of entries of $\mathcal{M}_n^{-1}(z)$ is $r_n(z)=\frac{1}{z+1}+k_n(z)$, where $k_n(z)$ is the sum of entries of $K_n^{-1}(z)$. For each $i$, we have 
	\begin{equation}\label{eq:inverse}
		\sum_{k=1}^{(n-1)^2}(K_n)_{i,k}\sum_{j=1}^{(n-1)^2}(K_n^{-1})_{k,j}=\sum_{j=1}^{(n-1)^2}\delta_{i,j}=1.
	\end{equation}
	Let $C_{n,i}$ denote the $i^{\text{th}}$ column sum of $K_n$, and $R_{n,i}^{-1}$ denote the $i^{\text{th}}$ row sum of $K_n^{-1}$. Then by adding Equation~\eqref{eq:inverse} for each $i=1,\dots,(n-1)^2$, and using the fact that $C_{n,i}=z+n-1$, for each $i$, we have 
	\[
	(n-1)^2=\sum_{i=1}^{(n-1)^2}C_{n,i}R_{n,i}^{-1}= (z+n-1)\sum_{i=1}^{(n-1)^2}R_{n,i}^{-1}=(z+n-1)k_n(z).
	\]
	Hence $r_n(z)=\frac{1}{z+1}+\frac{(n-1)^2}{z+n-1}$. Solving $z-n+r_n(z)=0$, we obtain that the Perron root is given by the largest root of $z^2-(n-1)$ which is $\sqrt{n-1}$.
\end{exam}

\section{Intermediate results}
Let $\Sigma_\F$ be an irreducible subshift of finite type where $\Sigma=\{0,1,\dots,q-1\}$ and $\F=\{a_1,\dots,a_s\}$ is a minimal collection with the longest word having length $p\ge 2$. From~\eqref{eq:solF}, 
\[
F(z) = \dfrac{z}{z-q+r(z)}, 
\]
where $r(z)$ is the sum of the entries of $\mathcal{M}(z)^{-1}$. In this section, we will discuss some properties of the rational function $r(z)$ (consequently $F(z)$) which will be used in due course.

\begin{lemma}\label{lemma:r_sing}
	The rational function $r$ is either analytic or has a removable singularity at $\theta$ with $r(\theta)=q-\theta$.
\end{lemma}
\begin{proof}
	Since $\theta$ is a pole of $F$, $\theta$ is not a pole for the rational function $r$. Thus $r(\theta)$ is defined and equals $q-\theta$ which is positive.
\end{proof}

\noindent Consequently, $r'(z)$ exists in a neighborhood of $\theta$. 

\begin{lemma}\label{lemma:r'}
	$1+r'(\theta)>0$.
\end{lemma}
\begin{proof}
	Since $\theta$ is a simple pole for $F$, $(z-\theta)F(z)$ is analytic at $z=\theta$ and $\lim_{z\rightarrow\theta}(z-\theta)F(z)=\lim_{z\rightarrow\theta^+,   z\in\mathbb{R}} (z-\theta)F(z)>0$. Moreover, 
	\[
	\lim_{z\rightarrow\theta} (z-\theta)F(z)=\lim_{z\rightarrow\theta} \dfrac{(z-\theta)z}{z-q+r(z)}= \lim_{z\rightarrow\theta} \dfrac{2z-\theta}{1+r'(z)}.
	\]
	Hence $1+r'(\theta)>0$.
\end{proof}
Using Theorem~\ref{thm:formF}, $\sum_i F_i(z)=1-\dfrac{(z-q)}{z}F(z)$. Since by Theorem~\ref{thm:same_roots}, $\theta$ is a simple pole for $F$, and $F_i(z)>0$ for all $z>0$, each $(z-\theta)F_i(z)$ is analytic at $\theta$. Therefore, 
\[
\lim_{z\rightarrow \theta} (z-\theta)F_i(z)= \lim_{z\rightarrow \theta} \dfrac{(z-\theta)F(z)}{z}\mathcal{R}_i(z) 
\]
exists, where $\mathcal{R}_i(z)$ is the sum of the entries of the $i^{th}$ row of $\mathcal{M}(z)^{-1}$. Since $\lim_{z\rightarrow \theta} \mathcal{R}_i(z)= \lim_{z\rightarrow \theta}\dfrac{z(z-\theta)F_i(z)}{(z-\theta)F(z)}$ where the limits of both numerator and denominator exist and the limit of the denominator is positive, $\lim_{z\rightarrow \theta}\mathcal{R}_i(z)$ exists.

Similarly using Theorems~\ref{thm:formG} and~\ref{thm:same_roots}, each $(z-\theta)G_j(z)$ is analytic at $\theta$, hence, 
\[
\lim_{z\rightarrow \theta}(z-\theta)G_j(z)  = \lim_{z\rightarrow \theta} \dfrac{(z-\theta)F(z)}{z}\mathcal{C}_j(z),
\] 
exists, where $\mathcal{C}_j(z)$ is the sum of the entries of the $j^{th}$ column of $\mathcal{M}(z)^{-1}$. Also $\lim_{z\rightarrow \theta}\mathcal{C}_j(z)$ exists.

\begin{lemma}\label{lemma:RC}
	The limits $\lim_{z\rightarrow \theta}\mathcal{R}_i(z)$ and $\lim_{z\rightarrow \theta}\mathcal{C}_j(z)$ exist, for all $i,j=1,\dots,s$.
\end{lemma}

\noindent Since the limits exist, denote
\[
\mathcal{R}_i(\theta):=\lim_{z\rightarrow \theta}\mathcal{R}_i(z),\ \ \mathcal{C}_j(\theta):=\lim_{z\rightarrow \theta}\mathcal{C}_j(z).
\]

Let $w$ be an allowed word of length $n\ge p$ in $\Sigma_\F$. Consider the correlation matrix functions $\mathcal{M}(z)$ for the collection $\F$, and $\mathcal{M}_w(z)$ for the collection $\F\cup\{w\}$ given by
\[
\mathcal{M}_w(z)=\begin{pmatrix}
	\mathcal{M}(z)&X(z)\\ Y(z)&Z(z)
\end{pmatrix},
\]
where $X(z)=((w,a_1)_z,\dots,(w,a_s)_z)^T$, $Y(z)=((a_1,w)_z,\dots,(a_s,w)_z)$, and $Z(z)=(w,w)_z$. Also,
\begin{itemize}
	\item as earlier, $r(z)=\mathcal{S}(z)/\mathcal{D}(z)$, where $\mathcal{S}(z)$ denotes the sum of the entries of the adjoint matrix of $\mathcal{M}(z)$ and $\mathcal{D}(z)$ denotes the determinant of $\mathcal{M}(z)$. Recall that $\theta$ is the largest positive real zero in modulus of $(z-q)+r(z)$, which coincides with the Perron root of the adjacency matrix for $\Sigma_\F$.
	\item let $r_w(z)=\mathcal{S}_w(z)/\mathcal{D}_w(z)$, where $\mathcal{S}_w(z)$ denotes the sum of the entries of the adjoint matrix of $\mathcal{M}_w(z)$ and $\mathcal{D}_w(z)$ denotes the determinant of $\mathcal{M}_w(z)$. Note that $\Sigma_{\F\cup\{w\}}$ need not be irreducible. However, if $\theta_w$ denotes the Perron root of the adjacency matrix for $\Sigma_{\F\cup\{w\}}$, then $\theta_w$ is the largest positive real zero in modulus of $(z-q)+r_w(z)$, by Theorem~\ref{thm:same_roots}.
\end{itemize}

The following lemma is an easy consequence of the determinant and inverse formulae for block matrices, see~\cite{HJ_book}.

\begin{lemma}\label{lemma:mu_alt}
	The following holds true:
	\begin{eqnarray*}
		\lim_{z\rightarrow \theta} \dfrac{\mathcal{D}(z)\mathcal{S}_w(z)-\mathcal{S}(z)\mathcal{D}_w(z)}{\mathcal{D}(z)^2} &=& \left(1-\sum_{i=1}^s \mathcal{R}_i(\theta) (a_i,w)_\theta\right)\left(1-\sum_{j=1}^s \mathcal{C}_j(\theta) (w,a_j)_\theta\right).
	\end{eqnarray*}
\end{lemma} 

\begin{rem}\label{rem:choice_w}
	Let $w=w_1\dots w_n$ be an allowed word, $x=w_1\dots w_{p-1}$, and $y=w_{n-p+2}\dots w_n$. Since $(a_i,w)_z=(\widetilde{a_i},x)_z$, and $(w,a_j)_z=(y,a_j)_z$, for all $z$, and $i,j$, where $\widetilde{a_i}$ is the subword of $a_i$ obtained by removing the first symbol of $a_i$, Lemma~\ref{lemma:mu_alt} reduces to 
	\begin{eqnarray*}
		\lim_{z\rightarrow \theta} \dfrac{\mathcal{D}(z)\mathcal{S}_w(z)-\mathcal{S}(z)\mathcal{D}_w(z)}{\mathcal{D}(z)^2} &=& \left(1-\sum_{i=1}^s \mathcal{R}_i(\theta) (\widetilde{a_i},x)_\theta\right)\left(1-\sum_{j=1}^s \mathcal{C}_j(\theta) (y,a_j)_\theta\right).
	\end{eqnarray*}
	Hence the limit is independent of the word $w$ which starts with $x$ and ends with $y$.
\end{rem}

\begin{lemma}\label{lemma:mu_alt_nonzero}
	The limit obtained in Lemma~\ref{lemma:mu_alt} is positive. 
\end{lemma}
\begin{proof}
	For any allowed word $w$ which starts with $x$ and ends with $y$, 
	\begin{eqnarray}\label{eq:form}
		r_w(z)-r(z)&=& \dfrac{\mathcal{D}(z)\mathcal{S}_w(z)-\mathcal{S}(z)\mathcal{D}_w(z)}{\mathcal{D}(z)\mathcal{D}_w(z)}=\dfrac{\mathcal{D}(z)\mathcal{S}_w(z)-\mathcal{S}(z)\mathcal{D}_w(z)}{\mathcal{D}(z)^2}\dfrac{\mathcal{D}(z)}{\mathcal{D}_w(z)}.
	\end{eqnarray}
	Since $\theta_w$ is the Perron root corresponding to the adjacency matrix of $\Sigma_{\F\cup\{w\}}$, $\theta_w<\theta$. By the form of rational functions $r$ and $r_w$ in terms of $F$ and $F_w$, we get
	\begin{eqnarray}\label{eq:form1}
		\lim_{z\rightarrow\theta}\left(r_w(z)-r(z)\right)&=&\dfrac{\theta}{F_w(\theta)}>0.
	\end{eqnarray} 
	Further note that
	\[
	\dfrac{\mathcal{D}(z)}{\mathcal{D}_w(z)} = \dfrac{1}{(w,w)_z - \dfrac{Y(z) \text{Adjoint}(\mathcal{M}(z)) X(z)}{\mathcal{D}(z)}}.
	\]
	The limit (of the rational function) $\lim_{z\rightarrow\theta}\dfrac{Y(z) \text{Adjoint}(\mathcal{M}(z)) X(z)}{\mathcal{D}(z)}$ exists since otherwise $\lim_{z\rightarrow\theta}\dfrac{\mathcal{D}(z)}{\mathcal{D}_w(z)}=0$, which using~\eqref{eq:form} contradicts~\eqref{eq:form1}. \\
	Now since $\Sigma_\F$ is irreducible, choose a word $w$ sufficiently long such that 
	\[
	(w,w)_\theta>\lim_{z\rightarrow\theta} \dfrac{Y(z) \text{Adjoint}(\mathcal{M}(z)) X(z)}{\mathcal{D}(z)}.
	\]
	Thus $\lim_{z\rightarrow\theta}\dfrac{\mathcal{D}(z)}{\mathcal{D}_w(z)}$ exists and is positive. \\
	Taking limits on both sides of~\eqref{eq:form}, we obtain the desired result.
\end{proof}

\begin{rem}\label{rem:same_sign}
	Note that in the preceding proof, we chose a word $w$ satisfying a certain condition. However, the result is independent of this choice of $w$ due to Remark~\ref{rem:choice_w}. 
	Moreover, by Lemma~\ref{lemma:mu_alt_nonzero}, the terms in both brackets on the right of the expression in the statement of Lemma~\ref{lemma:mu_alt} are non-zero and have the same sign.
\end{rem}

\section{Left and right eigenvectors corresponding to the Perron root $\theta$}
In this section, we give an expression for eigenvectors of the adjacency matrix $A$ of an irreducible subshift $\Sigma_\F$ corresponding to the Perron root $\theta$ using correlation between the forbidden words. In what follows, the forbidden collection $\F=\{a_1,\dots,a_s\}$ is minimal with the longest word having length $p\ge 2$. The correlation matrix function of $\F$ is $\mathcal{M}(z)$, $\mathcal{D}(z)$ is the determinant of $\mathcal{M}(z)$, and $\mathcal{S}(z)$ is the sum of the entries of the adjoint matrix of $\mathcal{M}(z)$. For each allowed word $x$ of length $p-1$, let
\begin{eqnarray}\label{eq:ux_vy}
	u_x=1-\sum_{i=1}^s \mathcal{R}_i(\theta) (\widetilde{a_i},x)_\theta, \ \ \ v_x=1-\sum_{j=1}^s \mathcal{C}_j(\theta)(x,a_j)_\theta.
\end{eqnarray}

In Theorem~\ref{thm:eig_vec}, we will prove that the vectors $u=(u_x)_x$ and $v=(v_x)_x$ are left and right eigenvectors (not necessarily normalized) of the adjacency matrix $A$, corresponding to the Perron root $\theta$. 

\begin{rem}
	If $\Sigma_\F$ is irreducible, then for each pair of allowed words $x,y$, $u_xv_y>0$ by Lemma~\ref{lemma:mu_alt_nonzero}. Hence, all of the quantities in the set 
	\[
	\{u_x,v_y\ \vert\ x,y \text{ allowed words of length $p-1$ in $\Sigma_\F$}\}
	\]
	obtained in Theorem~\ref{thm:eig_vec} will have the same sign, which is also true by the Perron-Frobenius theorem. 
\end{rem}

In the following result, we consider the simplest case when $\F$ has a single forbidden word $a_1$ of length $p$. In this case, all the words of length $p-1$ are allowed in $\Sigma_F$. For each $x\in \Sigma^{p-1}$, 
\[
u_x=1-\dfrac{(\widetilde{a_1},x)_\theta}{(a_1,a_1)_\theta}, \ v_x=1-\dfrac{(x,a_1)_\theta}{(a_1,a_1)_\theta}. 
\]   

\begin{lemma}\label{lemma:one_word}
	Consider an irreducible subshift $\Sigma_\F$ where $\F=\{a_1\}$. Then the vectors $v=(v_x)_x$ and $u=(u_x)_x$ are respectively right and left eigenvectors of the adjacency matrix, corresponding to the Perron root $\theta$.
\end{lemma}
\begin{proof}
	Let $\F=\{a_1=a_{11}a_{12}\dots a_{1p}\}$ and $A$ be the corresponding adjacency matrix.
	We prove that $v$ is a right Perron eigenvector. Using similar arguments it can be shown that $u$ is a left Perron eigenvector.
	
	Note that any word of length $p-1$ is allowed in $\Sigma_\F$. Let $x=x_1\dots x_{p-1}$ and $y=y_1\dots y_{p-1}$ be two words of length $p-1$. 
	Recall that $A_{xy}=1$ if and only if $x_2=y_1,\dots, x_{p-1}=y_{p-2}$ and $x_1\dots x_{p-2}y_{p-1}\ne a_1$. Hence, in each row, except the row indexed by the word $a_{11}\dots a_{1(p-1)}$, there are exactly $q$ 1's. \\~\\
	\textbf{Case 1}: First consider a word, say $x=x_1x_2\dots x_{p-1}$, different from $a_{11}\dots a_{1(p-1)}$. As discussed, there are $q$ 1's in the row indexed by the word $x$, and they are of the form $A_{xx_\beta}$, where $x_\beta=x_2\dots x_{p-1}\beta$, for $\beta=0,1,\dots,q-1$. We need to show that $\sum_{\beta=0}^{q-1}v_{x_\beta}=\theta v_{x}$. That is, 
	\[
	q(a_1,a_1)_\theta-\sum_{\beta=0}^{q-1}(x_\beta, a_1)_\theta=\theta((a_1,a_1)_\theta-(x,a_1)_\theta),
	\] 
	since $(a_1,a_1)_\theta\ne 0$. This is true if and only if 
	\[
	(a_1,a_1)_\theta(q-\theta)=\sum_{\beta=0}^{q-1}(x_\beta,a_1)_\theta-\theta (x,a_1)_\theta.
	\] 
	By Lemma~\ref{lemma:r_sing} (notice that it does not require irreducibility of $\Sigma_\F$),  $(q-\theta)(a_1,a_1)_\theta=1$. Hence it is enough to show that 
	\[
	\sum_{\beta=0}^{q-1}(x_\beta,a_1)_\theta=\theta (x,a_1)_\theta+1.
	\] 
	We will, in fact, show that for all $z$, 
	\begin{eqnarray}\label{eq:corre}
		\sum_{\beta=0}^{q-1}(x_\beta,a_1)_z=z(x,a_1)_z+1.
	\end{eqnarray}
	Let $(x,a_1)_z=\sum_{j=1}^{p-2}b_{j}z^{p-2-j}$, and $(x_\beta,a_1)_z=\sum_{j=0}^{p-2}b_{\beta,j}z^{p-2-j}$, where each $b_{j},b_{\beta,j}$ is either 0 or 1. First observe that $b_{\beta,p-2}=1$ if and only if $\beta=a_{11}$. Hence, ~\eqref{eq:corre} is equivalent to proving that for all $j=1,\dots, p-2$,
	\[
	\left(\sum_{\beta=0}^{q-1}b_{\beta,j-1}\right) - b_{j}=0.
	\] 
	This is immediate since $b_{j}=1$ if and only if $x_{j+1}=a_{11},\dots,x_{p-1}=a_{1(p-j-1)}$. Moreover, $b_{\beta,j-1}=1$ if and only if $x_{j+1}=a_{11},\ \dots,\ x_{p-1}=a_{1(p-j-1)}$, and $\beta=a_{1(p-j)}$.\\~\\
	\textbf{Case 2}: Now consider $x=a_{11}\dots a_{1(p-1)}$. The row indexed by $x$ has exactly $q-1$ 1's. We need to show that 
	\[	(q-1)(a_1,a_1)_\theta-\sum_{\beta=0,\beta\ne a_{1p}}^{q-1}(x_\beta,a_1)_\theta=\theta(a_1,a_1)_\theta-\theta(x,a_1)_\theta.
	\] Using $(q-\theta)(a_1,a_1)_\theta=1$, this is equivalent to 
	\begin{equation*}
		\sum_{\beta=0,\beta\ne a_{1p}}^{q-1}(x_\beta,a_1)_\theta+(a_1,a_1)_\theta=\theta(x,a_1)_\theta+1,	
	\end{equation*}
	which is true by similar arguments as in Case 1. 
\end{proof} 

\begin{rem}
	Let $\Sigma=\{0,1,\dots,q-1\}$, $\F$ be minimal collection of words with the longest word having length $p\ge 2$. Let $x=x_1x_2\dots x_{p-1}$ be an allowed word of length $p-1$ in $\Sigma_\F$. Fix a forbidden word $a\in\F$. Consider the following two cases:\\
	i) If $x$ is such that for any $\beta\in\Sigma$, the word $x\beta=x_1x_2\dots x_{p-1}\beta$ of length $p$ does not end with $a$, then 
	\begin{eqnarray}\label{eq:corr}
		\sum_{\beta=0}^{q-1}(x_\beta,a)_z=z(x,a)_z+1,
	\end{eqnarray}
	where $x_\beta=x_2\dots x_{p-1}\beta$. \\	
	ii) If there exists a $\beta_0\in\Sigma$ such that $x\beta_0$ ends with $a$, then $x\beta_0=x_1x_2\dots x_{p-1}\beta_0$ is not allowed in $\Sigma_\F$. Moreover since $\F$ is a minimal collection, for these $x,\beta_0$, the forbidden word $a$ is unique. Also
	\begin{equation}\label{eq:corr1}
		\sum_{\beta=0,\ \beta\ne\beta_0}^{q-1}(x_\beta,a)_z+(a,a)_z=z(x,a)_z+1.	
	\end{equation}
	When $a$ has length $p$, the validity of Equation~\eqref{eq:corr} is covered in the proof of Theorem~\ref{lemma:one_word} and a similar argument can be used to prove Equation~\eqref{eq:corr1}.\\     
	When the length of $a$ is at most $p-1$, consider a new word $x'=x_{p-k+1}\dots x_{p-1}$ of length $k-1$ where $|a|=k$. Note that $x$ ends with $x'$. Since $x$ is allowed,
	\[(x,a)_z=(x',a)_z, \ \ (x_\beta,a)_z=(x_\beta',a)_z,
	\]
	where $x_\beta=x_2\dots x_{p-1}\beta$ and $x'_\beta=x_{p-k+2}\dots x_{p-1}\beta$.	
\end{rem} 

Now we consider the general situation when $\Sigma_\F$ is an irreducible subshift of finite type where $\F=\{a_1,\dots,a_s\}$ is a minimal collection with the longest word having length $p\ge 2$. We prove that the vectors $v=(v_x)_x$ and $u=(u_x)_x$, as defined in~\eqref{eq:ux_vy}, are right and left eigenvectors, respectively, of the adjacency matrix $A$ corresponding to the Perron root $\theta$.

\begin{thm}\label{thm:eig_vec}
	The vectors $v=(v_x)_x$ and $u=(u_x)_x$ are right and left eigenvectors, respectively, of $A$ corresponding to the Perron root $\theta$.
\end{thm}
\begin{proof}
	Recall the symbol set $\Sigma=\{0,1,\dots,q-1\}$. Let $x=x_1x_2\dots x_{p-1}$ be an allowed word of length $p-1$ in $\Sigma_\F$. Let 
	\[
	B=\{\beta\in\Sigma\ \vert \ x\beta \text{ is allowed in }\Sigma_\F\}.
	\]
	Then $B$ is a non-empty subset of $\Sigma$ since $x$ is an allowed word in $\Sigma_\F$. Further for each $\beta\in B$, $x_\beta=x_2\dots x_{p-1}\beta$ is an allowed word in $\Sigma_\F$.
	\noindent We need to prove that
	\begin{eqnarray}\label{eq:1}
		\sum_{\beta\in B} v_{x_\beta}&=&\theta v_x,
	\end{eqnarray}	
	There are two cases:\\~\\
	\textbf{Case 1}: $B=\Sigma$. That is, $x$ is such that  $x\beta$ (hence $x_\beta$) is allowed in $\Sigma_\F$ for all $\beta\in\Sigma$. Consider
	\begin{eqnarray*}
		\left(\sum_{\beta=0}^{q-1} v_{x_\beta}\right)-\theta v_x&=& \left(\sum_{\beta=0}^{q-1} \left(v_{x_\beta}-1\right)\right)-\theta (v_x-1)+(q-\theta)\\
		&=& \left(\sum_{\beta=0}^{q-1} \left(v_{x_\beta}-1\right)\right)-\theta (v_x-1)+\sum_{j=1}^s \mathcal{C}_j(\theta)\\
		&=& \sum_{j=1}^s \mathcal{C}_j(\theta)\left(-\sum_{\beta=0}^{q-1}(x_\beta,a_j)_\theta+\theta(x,a_j)_\theta+1\right),
	\end{eqnarray*}
	which equals 0 using~\eqref{eq:corr} (the second equality holds true since $q-\theta=r(\theta)=\sum_{j=1}^s \mathcal{C}_j(\theta)$). Hence~\eqref{eq:1} follows. \\~\\
	\textbf{Case 2}: $B\neq\Sigma$. Let $\Sigma\setminus B=\{\delta_1,\dots,\delta_t\}$ and let $x\delta_1,\dots,x\delta_t$ end with the forbidden words $a_1,\dots,a_t$, respectively. Notice that for each $i=1,\dots,t$, the word $a_i\in\F$ is unique for $\delta_i$ since $\F$ is a minimal collection.	As was done in the previous case, 
	\begin{eqnarray}\label{eq:2}
		\left(\sum_{\beta\in B} v_{x_\beta}\right)-\theta v_x
		&=& \left(\sum_{\beta\in B} \left(v_{x_\beta}-1\right)\right)-\theta (v_x-1)+(q-\theta)-t\nonumber\\
		&=& \sum_{j=1}^s \mathcal{C}_j(\theta)\left(-\sum_{\beta\in B}(x_\beta,a_j)_\theta+\theta(x,a_j)_\theta+1\right)-t.
	\end{eqnarray}
	When $j=1,\dots,t$, using~\eqref{eq:corr1}, we obtain
	\begin{equation}\label{eq:4}
		1+\theta(x,a_j)_\theta=(a_j,a_j)_\theta+\sum_{\beta=0,\ \beta\ne \delta_j}^{q-1}(x_\beta,a_j)_\theta,
	\end{equation}
	and when $j=t+1,\dots,s$, using~\eqref{eq:corr},
	\begin{eqnarray}\label{eq:5}
		1+\theta(x,a_j)_\theta&=&\sum_{\beta=0}^{q-1} (x_\beta,a_j)_\theta.
	\end{eqnarray}
	Using~\eqref{eq:4} and~\eqref{eq:5}, from~\eqref{eq:2},
	\begin{eqnarray}\label{eq:3}
		& & \sum_{j=1}^s \mathcal{C}_j(\theta)\left(-\sum_{\beta\in B}(x_\beta,a_j)_\theta+\theta(x,a_j)_\theta+1\right)\nonumber\\
		&=& \sum_{j=1}^t \mathcal{C}_j(\theta)\left((a_j,a_j)_\theta+\sum_{i=1,\ i\ne j}^{t}(x_{\delta_i},a_j)_\theta\right)+ \sum_{j=t+1}^s \mathcal{C}_j(\theta)\left(\sum_{i=1}^t (x_{\delta_i},a_j)_\theta\right)\nonumber\\
		&=& \sum_{j=1}^t \mathcal{C}_j(\theta)\left(\sum_{i=1}^{t} (a_i,a_j)_\theta\right)+ \sum_{j=t+1}^s \mathcal{C}_j(\theta)\left(\sum_{i=1}^t (a_i,a_j)_\theta\right)\nonumber\\
		&=& \sum_{i=1}^t \sum_{j=1}^s (a_i,a_j)_\theta \mathcal{C}_j(\theta),
	\end{eqnarray}
	since $(x_{\delta_i},a_j)_\theta=(a_i,a_j)_\theta$, for all $i\ne j$.\\~\\
	Finally, $\sum_{j=1}^s (a_i,a_j)_\theta \mathcal{C}_j(\theta)=\sum_{j=1}^s \mathcal{M}_{j,i}(\theta)\sum_{k=1}^s \mathcal{M}^{-1}_{k,j}(\theta)= \sum_{k=1}^s\delta_{i,k}=1$, for all $i=1,\dots,t$, and thus~\eqref{eq:3} equals $t$, which further implies that~\eqref{eq:2} equals 0. 
\end{proof}

\begin{rem}
	This proof of Theorem~\ref{thm:eig_vec} does not assume the irreducibility of $A$. Hence whenever $\mathcal{R}_i(\theta), \mathcal{C}_i(\theta)$ are well defined, Theorem~\ref{thm:eig_vec} gives an
	expression for a left/right Perron vector of the adjacency matrix of any subshift of finite type. However, all the entries of $u$, also $v$, need not be positive when the adjacency matrix is reducible.
\end{rem}

\section{Normalizing factor for Perron eigenvectors}
In this section, we use the concept of the local escape rate from ergodic theory to find the normalizing factor for the eigenvectors obtained in Theorem~\ref{thm:eig_vec}. The results in this section require $\Sigma_\F$ to be an irreducible subshift with positive topological entropy, that is, the Perron root is strictly bigger than one.

We will now define the concept of escape rate in the setting of a subshift of finite type. The notion of escape rate is more general than this but will not be needed in full generality for this work. Thus we will restrict ourselves to the subshifts of finite type.

\begin{Def}
	Consider an irreducible subshift of finite type $\Sigma_\F$. Let $\mathcal{G}$ be another non-empty finite collection of allowed words from $\Sigma_\F$. Consider the hole $H_{\mathcal{G}}=\bigcup_{w\in\mathcal{G}}C_w$ in $\Sigma_{\F}$, where $C_w\subseteq\Sigma_\F$ denotes the cylinder based at $w$. The \emph{escape rate} denotes the rate at which the orbits escape into the hole and is defined as 
	\[
	\rho(H_{\mathcal{G}}) := -\lim_{n\rightarrow \infty} \dfrac{1}{n} \ln\mu(\mathcal{W}_n(\mathcal{G})),
	\]
	if the limit exists, where $\mathcal{W}_n(\mathcal{G})$ denotes the collection of all sequences in $\Sigma_{\F}$ which do not include words from $\mathcal{G}$ as subwords in their first $n$ positions, and $\mu$ denotes the Parry measure.
\end{Def} 

\noindent In the given setting, the limit exists and is given by the following result.\\

\begin{thm}~\cite[Theorem 3.1]{Product}\label{thm:esc_rate}
	The escape rate into the hole $H_\mathcal{G}$ satisfies $\rho(H_{\mathcal{G}})= \ln (\theta/\lambda)>0$, where $\ln\theta$ and $\ln\lambda$ are topological entropies of $\Sigma_{\F}$ and $\Sigma_{\F\cup \mathcal{G}}$, respectively. 
\end{thm}

Now we define the concept of the local escape rate. Let $\alpha=\alpha_1\alpha_2\dots$ be a sequence in $\Sigma_{\F}$. The \emph{local escape rate} around $\alpha$ is defined as 
\[
\rho(\alpha)=\lim_{n\to\infty}\frac{\rho(H_{\F_n})}{\mu(H_{\F_n})},
\]
if it exists, where $\F_n=\{w^n=\alpha_1\alpha_2\dots \alpha_n\}$ and $H_{\F_n}=C_{w^n}$, the cylinder in $\Sigma_\F$ based at the word $w^n$. Note that $\bigcap_n H_{\F_n}=\{\alpha\}$. In~\cite{gibbs}, Ferguson and Pollicott gave an explicit formula for the local escape rate for a subshift of finite type. We now state their result in our setting.\\

\begin{thm}~\cite[Corollary 5.4.]{gibbs}\label{thm:local_erate}
	Let $\Sigma_\F$ be an irreducible subshift of finite type with positive topological entropy $\ln\theta$. Let $\alpha\in\Sigma_{\F}$. Then 
	\[ \rho(\alpha)=\begin{cases} 
		1 & \text{if $\alpha$ is non-periodic,} \\
		1-\theta^{-m} & \text{if $\alpha$ is periodic with period $m$}.
	\end{cases}
	\]
\end{thm} 

\noindent Using this, we obtain a relationship between the local escape rate around $\alpha$ and the auto-correlation polynomial of $w^n$.

\begin{lemma}\label{lemma:local_erate}
	Let $\Sigma_\F$ be an irreducible subshift of finite type with positive topological entropy $\ln\theta$. Let $\alpha=\alpha_1\alpha_2\dots\in\Sigma_{\F}$, and $w^n=\alpha_1\alpha_2\dots \alpha_n$, for all $n\ge 1$. Then
	\[
	\lim_{n \to \infty}\theta^{-n+1}(w^n,w^n)_\theta=\frac{1}{ \rho(\alpha)}.
	\] 
\end{lemma} 
\begin{proof}
	Let the autocorrelation polynomial of $w^n$ be
	\[
	(w^{n},w^{n})_z=z^{n-1}+\sum_{i=1}^{n-1}b_{n,i}z^{n-1-i},
	\]
	where $b_{n,i}$ is either 0 or 1. When $\alpha$ is non-periodic, $\lim_{n \to \infty}b_{n,i}=0$ for all $i$, and thus $\lim_{n \to \infty}\sum_{i=1}^{n-1}b_{n,i}\theta^{-i}=0$. Similarly when $\alpha$ is periodic with period $m$, $\lim_{n \to \infty} \sum_{i=1}^{n-1}b_{n,i}\theta^{-i}=\sum_{k=1}^\infty \theta^{-km}=\frac{1}{1-\theta^{-m}}-1$. By 
	Theorem~\ref{thm:local_erate}, in both the cases, $\lim_{n \to \infty}\theta^{-n+1}(w^n,w^n)_\theta=\frac{1}{ \rho(\alpha)}$.
\end{proof}

\begin{rem}\label{rem:g_alpha}
	Following the proof of the previous result, for $|z|>1$, \[\lim_{n \to \infty}z^{-n+1}(w^n,w^n)_z=g_\alpha(z),\] where 
	\[ 
	g_\alpha(z)=\begin{cases} 
		1 & \text{if $\alpha$ is non-periodic,} \\
		\left(1-z^{-m}\right)^{-1} & \text{if $\alpha$ is periodic with period $m$}.
	\end{cases}
	\]
	Also, $g_\alpha(\theta)=\frac{1}{\rho(\alpha)}$.
\end{rem}

The next result gives the normalizing factor for the eigenvectors obtained in Theorem~\ref{thm:eig_vec}.

\begin{thm}\label{thm:eig_prod}
	With notations as in Theorem~\ref{thm:eig_vec}, if $\Sigma_{\F}$ is an irreducible subshift with positive topological entropy, then 
	\begin{eqnarray*}
		u^Tv&=&\theta^{p-1} \left(1+r'(\theta)\right),
	\end{eqnarray*}
	where $p\ge 2$ is the length of longest word in the minimal collection $\F$.
\end{thm}

\begin{proof}
	The expression on the right is defined by Lemma~\ref{lemma:r'}. Let $\F=\{a_1,\dots,a_s\}$ be the minimal collection of forbidden words. Let $x,y$ be two allowed words each of length $p-1$ in $\Sigma_\F$. Let $w$ be a fixed allowed word in $\Sigma_\F$ beginning with $x$ and ending with $y$. Consider a finite word $\gamma$ such that $y\gamma y$ is an allowed word in $\Sigma_{\F}$. Since the subshift is irreducible, the words $w$ and $\gamma$ exist. Define $\alpha=\alpha_1\alpha_2\dots=w\overline{\gamma y}$, where $\overline{\gamma y}$ denotes the word $\gamma y$ repeated infinitely many times. Clearly $\alpha\in\Sigma_{\F}$. Let $\F_n=\{w^n=\alpha_1\dots\alpha_n\}$ and $H_{\F_n}=C_{w^n}$, the cylinder in $\Sigma_\F$ based at the word $w^n$. \\ 
	Let $\mathcal{M}_n(z)$ be the correlation matrix function for the collection $\F\cup\F_n$, and $r_{n}(z)=\mathcal{S}_n(z)/\mathcal{D}_n(z)$, where $\mathcal{D}_n(z)$ denotes the determinant of $\mathcal{M}_{n}(z)$ and $\mathcal{S}_n(z)$ denotes the sum of the entries of the adjoint matrix of $\mathcal{M}_n(z)$. \\
	Let $\lambda_n$ be the Perron root of the adjacency matrix corresponding to $\Sigma_{\F\cup\F_n}$. Then
	\begin{equation}\label{eq:r(theta)}
		q-\lambda_n=r_n(\lambda_n), \text{ and } q-\theta=r(\theta).
	\end{equation} 
	Choose a subsequence $\{n_k\}_{k\ge 0}$ such that $w^{n_k}=\alpha_1\dots\alpha_{n_k}=w\overline{\gamma y}^{k}$, for all $k\ge 0$, where $\overline{\gamma y}^{k}$ denotes the word $\gamma y$ repeated $k$ times, that is, the subsequence, where $\alpha_1\dots\alpha_{p-1}=x$ and $\alpha_{n_k-p+2}\dots\alpha_{n_k}=y$, for all $k\ge 0$. By Remark~\ref{rem:choice_w}, observe that for all $k\ge 0$,
	\begin{eqnarray}\label{r_n-r}
		r_{n_k}(z)-r(z)&=&\dfrac{\mathcal{D}(z)\mathcal{S}_{n_k}(z)-\mathcal{D}_{n_k}(z)\mathcal{S}(z)}{\mathcal{D}(z)\mathcal{D}_{n_k}(z)}=\dfrac{\mathcal{D}(z)\mathcal{S}_w(z)-\mathcal{D}_w(z)\mathcal{S}(z)}{\mathcal{D}(z)\mathcal{D}_{n_k}(z)}.
	\end{eqnarray}
	Let $E$ be a deleted neighbourhood of $\theta$ such that $|z|>1$ for all $z\in E$ and $\mathcal{D}(z)\ne 0$ for all $z\in E$. Such a neighbourhood exists since $\theta>1$ and $\mathcal{D}(z)$ is a non-constant polynomial.\\
	\noindent Since \begin{eqnarray*}
		\mathcal{D}_{n_k}(z)&=&\mathcal{D}(z)(w^{n_k},w^{n_k})_z - ((a_1,w^{n_k})_z,\dots,(a_s,w^{n_k})_z) \\
		& & \text{Adjoint}(\mathcal{M}(z)) ((w^{n_k},a_1)_z,\dots,(w^{n_k},a_s)_z)^T\\ 
		&=& \mathcal{D}(z)(w^{n_k},w^{n_k})_z - ((a_1,w)_z,\dots,(a_s,w)_z) \\
		& & \text{Adjoint}(\mathcal{M}(z)) ((w,a_1)_z,\dots,(w,a_s)_z)^T,
	\end{eqnarray*}
	$\mathcal{D}_{n_k}(z)-\mathcal{D}(z)(w^{n_k},w^{n_k})_z$ is a polynomial, independent of the sequence $n_k$. Hence $\mathcal{D}_{n_k}(z)$ is a monic polynomial of degree at least $n_k$. Thus
	\[
	\lim_{k \to \infty}\frac{1}{\mathcal{D}_{n_k}(z)}=0, 
	\] uniformly on $E$.\\
	Therefore $r_{n_k}(z)$ converges to $r(z)$ uniformly in $E$ as $k$ tends to infinity.
	Hence $\lim_{k\to\infty} r_{n_k}(\theta)=r(\theta)$.\\
	Using the mean value theorem for derivatives,
	\begin{equation}\label{eq:r_n(theta)}
		q-\lambda_n=r_n(\lambda_n)=r_n(\theta)+(\lambda_n-\theta)r_n'(a_n),
	\end{equation}
	for some $\lambda_n<a_n<\theta$. 
	Taking the difference of~\eqref{eq:r_n(theta)} and~\eqref{eq:r(theta)}, we obtain 
	\[
	\theta-\lambda_n=r_n(\theta)-r(\theta)+(\lambda_n-\theta)r_n'(a_n),
	\]
	which gives
	\[
	\theta-\lambda_n=\frac{r_n(\theta)-r(\theta)}{1+r_n'(a_n)}.
	\]
	Set 
	\begin{eqnarray}\label{eq:Kn}
		K_n&=&\frac{r_n(\theta)-r(\theta)}{\theta\left(1+r_n'(a_n)\right)}=\lim_{z\rightarrow\theta}\frac{r_n(z)-r(z)}{z\left(1+r_n'(a_n)\right)}.
	\end{eqnarray}
	Then $K_n=1-\lambda_n/\theta$. The escape rate into the hole $H_{\F_n}$ is given by $\rho(H_{\F_n})=\ln\left(\theta/\lambda_n\right)$ by Theorem~\ref{thm:esc_rate}. Hence
	\begin{eqnarray}\label{eq:rKn}
		\rho(H_{\F_n})= -\ln(1-K_n)=K_n+\frac{K_n^2}{2}+\frac{K_n^3}{3}+\dots.
	\end{eqnarray}
	From Equation~\eqref{r_n-r}, it is clear that $(\mathcal{D}_{n_k}(r_{n_k}-r))_k$ is a constant sequence of functions. Using this fact and differentiating Equation~\eqref{r_n-r} with respect to $z$, we get $\lim_{k\to\infty}r_{n_k}'(z)=r'(z)$ uniformly on $E$.\\  
	Moreover, since $\lambda_n\rightarrow \theta$,
	\begin{eqnarray}\label{eq:r'limit}
		\lim_{k\to\infty}r_{n_k}'(a_{n_k})=\lim_{k\to\infty}r_{n_k}'(\theta)=r'(\theta),
	\end{eqnarray}
	which is not equal to $-1$, by Lemma~\ref{lemma:r'}.
	Observe that for $|z|>1$, 
	\[\lim_{k \to \infty}z^{-n_k+1}\mathcal{D}_{n_k}(z)=\lim_{k \to \infty}z^{-n_k+1}\mathcal{D}(z)(w^{n_k},w^{n_k})_z=\mathcal{D}(z)g_\alpha(z),
	\] 
	where $g_\alpha(z)=1$, if $\alpha$ is non-periodic, and $g_\alpha(z)=(1-z^{-m})^{-1}$, if $\alpha$ is periodic with period $m$, as in Remark~\ref{rem:g_alpha}. Hence
	\begin{eqnarray}\label{eq:r_w-r}
		\lim_{k \to \infty}\lim_{z\rightarrow \theta}z^{n_k-1}\left(r_{n_k}(z)-r(z)\right)&=&\lim_{k \to \infty}\lim_{z\rightarrow \theta}\frac{\mathcal{D}(z)\mathcal{S}_w(z)-\mathcal{D}_w(z)\mathcal{S}(z)}{z^{-n_k+1}\mathcal{D}_{n_k}(z)\mathcal{D}(z)}\nonumber\\
		&=&\lim_{z\rightarrow \theta}\frac{\mathcal{D}(z)\mathcal{S}_w(z)-\mathcal{D}_w(z)\mathcal{S}(z)}{g_\alpha(z)\mathcal{D}(z)^2}\nonumber\\
		&=&\rho(\alpha) u_xv_y.
	\end{eqnarray}
	Since the Perron eigenvectors are given by $u$ and $v$ as in Theorem~\ref{thm:eig_vec}, the vectors $U=u/\sqrt{u^Tv}$ and $V=v/\sqrt{u^Tv}$ are normalized Perron eigenvectors. Hence the Parry measure of the cylinder $H_{\F_n}$ is given by 
	\[
	\mu(H_{\F_n}) = \dfrac{U_xV_y}{\theta^{n-p+1}}=  \dfrac{u_{x}v_y}{\theta^{n-p+1}u^Tv}.
	\]
	Now using~\eqref{eq:Kn},~\eqref{eq:r'limit} and~\eqref{eq:r_w-r},
	\begin{eqnarray}\label{eq:limitKn}
		\lim_{k \to \infty}\frac{K_{n_k}}{\mu(H_{\F_{n_k}})}	&=&
		\lim_{k \to \infty}\frac{K_{n_k}\theta^{{n_k}-p+1}u^Tv}{u_xv_y}\nonumber\\
		&=&\dfrac{u^Tv}{u_xv_y}\lim_{k \to \infty}\lim_{z\rightarrow\theta}\frac{r_{n_k}(z)-r(z)}{z\left(1+r_{n_k}'(a_{n_k})\right)} z^{{n_k}-p+1}\nonumber\\
		&=&\dfrac{u^Tv}{\theta^{p-1}u_xv_y(1+r'(\theta))}\lim_{k \to \infty}\lim_{z\rightarrow\theta}(r_{n_k}(z)-r(z))z^{{n_k}-1}\nonumber\\
		&=& \rho(\alpha)\dfrac{u^Tv}{\theta^{p-1}\left(1+r'(\theta)\right)}.
	\end{eqnarray}
	Since $\lambda_n\rightarrow \theta$, $K_n\rightarrow 0$. Thus using~\eqref{eq:limitKn}, $\lim_{k \to \infty}K_{n_k}^j/\mu\left(H_{\F_{n_k}}\right)=0$, for all $j\ge 2$. Therefore, using Equation~\ref{eq:rKn},
	\begin{eqnarray*}
		\rho(\alpha) = \lim_{k \to \infty}\dfrac{\rho\left(H_{\F_{n_k}}\right)}{\mu\left(H_{\F_{n_k}}\right)}= \lim_{k \to \infty}\dfrac{K_{n_k}}{\mu\left(H_{\F_{n_k}}\right)}=\rho(\alpha)\dfrac{u^Tv}{\theta^{p-1}\left(1+r'(\theta)\right)}.
	\end{eqnarray*}
	Thus we have the desired result.
\end{proof}	

Let $A$ be a primitive matrix with $\theta>1$. Let $x,y$ be two allowed words of length $p-1$ and $f_{x,y}(n)$ be the number of words of length $n$ in $\Sigma_{\F}$, starting with $x$ and ending with $y$. Let the $xy^{th}$ entry of $A^n$ be denoted as $f_{x,y}(n+p-1)$. By the Perron-Frobenius theorem, 
\[
\lim_{n \to \infty} \frac{f_{x,y}(n+p-1)}{\theta^n}=V_{x}U_{y}=\dfrac{v_xu_y}{u^Tv},
\]
where $\theta$ is the Perron root of $A$, and $U$ and $V$ are normalized left and right Perron eigenvectors such that $U^TV=1$. Hence we obtain Corollary~\ref{cor:normal}.

\begin{cor}\label{cor:normal}
	With the notations as above, 
	\begin{eqnarray*}
		\lim_{n \to \infty}\frac{f_{x,y}(n)}{\theta^n}&=&\dfrac{\left(1-\sum\limits_{i=1}^s \mathcal{C}_j(\theta) (x,a_j)_\theta\right)\left(1-\sum\limits_{i=1}^s \mathcal{R}_i(\theta) (\widetilde{a_i},y)_\theta\right)}{\theta^{2p-2} \left(1+r'(\theta)\right)}.
	\end{eqnarray*} 
\end{cor} 

\section{Applications}\label{sec:app}

\subsection{An application to irreducible $(0,1)$ matrices}

\begin{thm}
	Let $A=[A_{xy}]_{1\le x,y\le n}$ be a $(0,1)$ irreducible matrix of size $n$ with Perron root $\theta$. Let $\F=\{xy \ \vert\ A_{xy}=0,\ 1\le x,y\le n\}$, labelled as $\{a_1,\dots, a_s\}$. Let $u=(u_x)_{1\le x\le n}$
	and $v=(v_x)_{1\le x\le n}$ be the vectors defined as 	
	\[
	u_x=1-\sum\limits_{\substack{i=1\\a_i\text{ ends with }x}}^s \mathcal{R}_i(\theta), \ \ v_x=1-\sum\limits_{\substack{i=1\\a_i\text{ begins with }x}}^s \mathcal{C}_i(\theta).
	\]
	Then for each $x$, $u_xv_x>0$. Moreover, $v$ and $u$ are right and left Perron eigenvectors, respectively, of $A$. \\
	Further, if $\theta>1$, the normalization factor for these Perron eigenvectors is given by 
	\[u^Tv =\theta(1+r'(\theta)).
	\]
\end{thm}
\begin{proof}
	Use Theorem~\ref{thm:eig_vec}. Observe that $(\tilde{a_i},x)_z=1$ if $a_i$ ends with $x$, else equals 0. Also $(x,a_i)_z=1$ if $a_i$ begins with $x$, else equals 0.
\end{proof}

\subsection{An alternate definition of the Parry measure}\label{subsec:parry}
An immediate consequence of Theorem~\ref{thm:eig_prod} is the following result which gives an alternate definition for the Parry measure~\eqref{eq:pm} as was stated in the beginning of the paper.

\begin{thm}\label{thm:parry}
	Let $\Sigma_\F$ be an irreducible subshift of finite type with positive topological entropy. Let $w$ be an allowed word in $\Sigma_\F$ of length $n\ge p$ which starts with $(p-1)$-word $x$ and ends with $(p-1)$-word $y$. Then 
	\begin{eqnarray}\label{eq:pm}
		\mu(C_w)&=&\dfrac{\left(1-\sum\limits_{i=1}^s \mathcal{R}_i(\theta) (\widetilde{a_i},x)_\theta\right)\left(1-\sum\limits_{j=1}^s \mathcal{C}_j(\theta) (y,a_j)_\theta\right)}{\theta^{n} \left(1+r'(\theta)\right)}.
	\end{eqnarray} 
\end{thm}
\begin{proof}
	By~\eqref{eq:pmold}$, \mu(C_w)=U_xV_y/\theta^{n-p+1}$, use Theorem~\ref{thm:eig_prod}.
\end{proof}

The expression thus obtained for $\mu(C_w)$ requires the Perron root $\theta$ (which can be obtained using Theorem~\ref{thm:same_roots}), the rational function $r$, the inverse of the correlation matrix function $\mathcal{M}(z)$ for the collection $\F$, and the correlation of the forbidden words from $\F$ with $x$ and $y$. This alternate definition highlights several properties about the Parry measure which are not evident from the original definition~\eqref{eq:pmold}.

\begin{rems}
	(1) The Parry measure of cylinders based at words of identical length with the same starting $(p-1)$-word and the same ending $(p-1)$-word is equal. This is also reflected in~\eqref{eq:pm}.\\
	(2) All cylinders based at words $w$ satisfying $(w,a_i)_\theta=(a_i,w)_\theta=0$, for each $i=1,\dots,s$, have same measure given by 
	\[
	\mu(C_w)=\dfrac{1}{\theta^{n} \left(1+r'(\theta)\right)}.
	\]
	(3) Further, if allowed words $w,w'$, both of the same length $n\ge p$ are such that $(a_i,w)_\theta=(a_i,w')_\theta$ and $(w,a_i)_\theta=(w',a_i)_\theta$, for all $i=1,\dots,s$, then $\mu(C_w)=\mu(C_{w'})$.\\
	(4) Also if $(a_i,a_j)_\theta=0$ for all $i\ne j$, then~\eqref{eq:pm} reduces to
	\begin{eqnarray*}
		\mu(C_w)&=&\dfrac{1}{\theta^{n} \left(1+r'(\theta)\right)} \left(1-\sum_{i=1}^s \dfrac{(\widetilde{a_i},x)_\theta}{(a_i,a_i)_\theta}\right)\left(1-\sum_{j=1}^s \dfrac{(y,a_j)_\theta}{(a_j,a_j)_\theta}\right),
	\end{eqnarray*}
	where $r(z)=\sum_{i=1}^s1/(a_i,a_i)_z$.\\
	(5) Finally, one of the greatest advantage of the new definition is that to compute the measure of a cylinder, one does not need to know the eigenvectors $U$ and $V$ in~\eqref{eq:pm}, as are required in~\eqref{eq:pmold}. 
\end{rems}

\section{Illustrative examples}

\begin{exam}
	Consider an irreducible subshift $\Sigma_\F$ with $\Sigma=\{0,1,2\}$ ($q=3$) and one forbidden word $\F=\{01\}$ ($p=2$). The correlation matrix is given by $\mathcal{M}(z)=(z)$, thus $r(z)=1/z$. Therefore the denominator of the generating function is $(z-q)+r(z)=(z-3)+1/z$, the largest root of which is the same as the Perron root $\theta$ (Theorem~\ref{thm:same_roots}). Using Theorem~\ref{thm:eig_vec}, for all $x\in\{0,1,2\}$,
	\[
	u_x=1- \dfrac{(1,x)_\theta}{\theta},\ v_x=1-\dfrac{(x,00)_\theta}{\theta}.
	\]
	Thus 
	\[
	u^T=(u_0,u_1,u_2)=(1,1-1/\theta,1), \ v^T=(v_0,v_1,v_2)=(1-1/\theta,1,1).
	\]
	Hence, $u^Tv=3-2/\theta$.\\
	On the other hand, using Theorem~\ref{thm:eig_prod}, 
	\[
	u^Tv = \theta(1+r'(\theta))= \theta-1/\theta.
	\]
	These two expressions for $u^Tv$ are consistent since $\theta$ is a root of $(z-3)+1/z=0$.
\end{exam}

\begin{exam}
	Consider an irreducible subshift $\Sigma_\F$ with $\Sigma=\{0,1,2\}$ ($q=3$) and one forbidden word $\F=\{00\}$ ($p=2$). The correlation matrix $\mathcal{M}(z)=(z+1)$, thus $r(z)=1/(z+1)$. Therefore the denominator of the generating function is $(z-q)+r(z)=(z-3)+1/(z+1)=0$, the largest root of which is $\theta=\sqrt{3}+1$, the same as the Perron root (Theorem~\ref{thm:same_roots}). Using Theorem~\ref{thm:eig_vec}, for all $x\in\{0,1,2\}$,
	\[
	u_x=1- \dfrac{(0,x)_\theta}{\theta+1},\ v_x=1-\dfrac{(x,00)_\theta}{\theta+1}.
	\]
	Thus 
	\[
	u^T=(u_0,u_1,u_2)=(1-1/(\theta+1),1,1), \ v^T=(v_0,v_1,v_2)=(1-1/(\theta+1),1,1).
	\]
	Further using Theorem~\ref{thm:eig_prod}, 
	\[
	u^Tv = \theta(1+r'(\theta))= \dfrac{\theta^2(\theta+2)}{(\theta+1)^2}.
	\]
	Thus using Theorem~\ref{thm:parry}, for any allowed word $w$ of length $n$ which begins with symbol $x$ and ends with symbol $y$,
	\begin{eqnarray*}
		\mu(C_w)&=&\dfrac{\left(1- \dfrac{(0,x)_\theta}{\theta+1}\right)\left(1-\dfrac{(y,00)_\theta}{\theta+1}\right)}{\theta^{n} \dfrac{\theta(\theta+2)}{(\theta+1)^2}}=\dfrac{\left(\theta+1-(0,x)_\theta\right)\left(\theta+1-(y,00)_\theta\right)}{\theta^{n+1}(\theta+2)}.
	\end{eqnarray*}
	As an illustration, for words $w$ of length two beginning with symbol $x$ and ending with symbol $y$, 
	\[
	\mu(C_{xy})= \dfrac{\theta+1}{\theta^2(\theta+2)}=\dfrac{3-\sqrt{3}}{12}, 
	\]
	for all pairs $(x,y)=(0,1)$, $(0,2)$, $(1,0)$, $(2,0)$. Similarly, for $(x,y)=(1,1)$, $(2,2)$, $(1,2)$, and $(2,1)$, $\mu(C_{xy}) =\sqrt{3}/12$. \\
	By directly computing the Perron root and corresponding left and right eigenvectors of the adjacency matrix $\begin{pmatrix}0&1&1\\1&1&1\\1&1&1\end{pmatrix}$ of the subshift, we get $\theta=\sqrt{3}+1$, $v=u=(\sqrt{3}-1,1,1)^T/c$, where $c^2=u^Tv=6-2\sqrt{3}=2\sqrt{3}(\sqrt{3}-1)$. Hence, for words $w$ of length two beginning with symbol $x$ and ending with symbol $y$, 	
	\[
	\mu(C_{xy}) =\dfrac{\sqrt{3}-1}{c^2\theta}=\dfrac{3-\sqrt{3}}{12},
	\]
	for all pairs $(x,y)=(0,1)$, $(0,2)$, $(1,0)$, $(2,0)$. A similar statement is true for the remaining pairs $(x,y)$.
\end{exam}

\begin{exam}
	Let $q=5$ and $\F=\{00,1010\}$. Here the length of the longest word of $\F$ is $p=4$ and the associated adjacency matrix is irreducible. We use Theorem~\ref{thm:parry} for finding the Parry measure of cylinders in $\Sigma_\F$. The correlation matrix for $\F$ is given by 
	\[
	\mathcal{M}(z)=\begin{pmatrix}
		z+1&1\\
		0&z^3+z
	\end{pmatrix},
	\]
	which gives $r(z)=\frac{z^2+2}{z^3+z^2+z+1}$. The largest positive real zero in modulus of $(z-5)+r(z)=0$ is $\theta\sim 4.82113$ (the Perron root). \\
	Using Theorem~\ref{thm:parry}, for any allowed word $w$ of length $n$ which begins with the word $x$ and ends with the word $y$, both of length three,
	\begin{eqnarray*}
		\mu(C_w)&=&\dfrac{1}{\theta^{n} \left(1-\dfrac{\theta^4+5\theta^2+2\theta+2}{(\theta^3+\theta^2+\theta+1)^2}\right)}\left(1-\dfrac{(\theta^3+\theta-1)(0,x)_\theta}{\theta^4+\theta^3+\theta^2+\theta}-\dfrac{(010,x)_\theta}{\theta^3+\theta}\right)\times\\
		&&\hspace{5cm} \left(1-\dfrac{(y,00)_\theta}{\theta+1}-\dfrac{(y,1010)_\theta}{\theta^3+\theta^2+\theta+1}\right).
	\end{eqnarray*} 
	In particular, if $w=0101$, then $x=010$, $y=101$. Substituting $(0,x)_\theta=1,(010,x)_\theta=\theta^2+1$, $(y,00)_\theta=0$, $(y,1010)_\theta=\theta^2+1$, we obtain $\mu(C_w)\sim0.000987$.
\end{exam}

\begin{exam}
	Let $q=5$ and $\F=\{0000,0001\}$. Here the length of the forbidden words is $p=4$ and the adjacency matrix is irreducible with size $5^3=125$. We use Theorem~\ref{thm:parry} for finding the Parry measure of cylinders. The correlation matrix for $\F$ is given by 
	\[
	\mathcal{M}(z)=\begin{pmatrix}
		z^3+z^2+z+1&0\\
		z^2+z+1&z^3
	\end{pmatrix},
	\]
	which gives $r(z)=\frac{2z^3}{z^6+z^5+z^4+z^3}$. The largest positive real zero in modulus of $(z-5)+r(z)=0$ is $\theta\sim 4.987$ (the Perron root). \\
	Using Theorem~\ref{thm:parry}, for any allowed word $w$ of length $n$ which begins with the word $x$ and ends with the word $y$ of length 3,
	\begin{eqnarray*}
		\mu(C_w)&=&\dfrac{1}{\theta^{n} \left(1-\dfrac{6\theta^2+4\theta+2}{(\theta^3+\theta^2+\theta+1)^2}\right)}\left(1-\dfrac{ \theta^3 ((000,x)_\theta+(001,x)_\theta)}{\theta^6+\theta^5+\theta^4+\theta^3}\right)\times\\
		&&\left(1-\dfrac{(\theta^3-\theta^2-\theta-1)(y,0000)_\theta+(\theta^3+\theta^2+\theta+1)(y,0001)_\theta}{\theta^6+\theta^5+\theta^4+\theta^3}\right).
	\end{eqnarray*} 
	In particular, if $w=0101$, then $x=010$, $y=101$. Substituting $(000,x)_\theta=1,(001,x)_\theta=\theta$, $(y,0000)_\theta=0$, $(y,0001)_\theta=0$, we obtain
	\begin{eqnarray*}
		\mu(C_w)&=&\dfrac{1}{\theta^{4} \left(1-\dfrac{6\theta^2+4\theta+2}{(\theta^3+\theta^2+\theta+1)^2}\right)}\left(1-\dfrac{\theta^3(\theta+1)}{\theta^6+\theta^5+\theta^4+\theta^3}\right)\sim 0.001565.
	\end{eqnarray*} 
\end{exam}

\section{Acknowledgements}
We thank the anonymous referees and the editor for their valuable comments and suggestions for improving the paper. Aditya Thorat provided an excellent research support and assistance with simulating examples in Python. The code is available in the next section. The funding support for this research are gratefully acknowledged. The research of the first author is supported by the Council of Scientific \& Industrial Research (CSIR), India (File no.~09/1020(0133)/2018-EMR-I), and the second author is supported by the Science Engineering Research Board, Department
of Science and Technology, India (File No. CRG/2019/003823) and the Center for Research on Environment and Sustainable Technologies (CREST), IISER Bhopal, CoE funded by the Ministry of Human Resource Development (MHRD), India.

\section{Python code}
\begin{verbatim}
	import numpy as np
	from itertools import product
	Q= input(" Enter the size of symbol set:")
	q= int(Q)
	S= input("Enter the number of forbidden words:")
	s= int(S)
	#Taking forbidden words as input
	forbiddenwords=['' for i in range(s)]
	for i in range(s):
	w= input(" Enter forbidden word no."+str(i)+":" )
	for j in range(len(w)):
	forbiddenwords[i]= forbiddenwords[i]+w[j]
	#Finding the length of the longest forbidden word
	L= [0 for i in range(s)]
	for i in range(s):
	L[i]= len(forbiddenwords[i])
	P= max(L)
	#Generating set of all words of length p-1 with symbols from {0,1,...,q-1}
	Allwords=[]
	alphabet=[]
	for i in range(q):
	alphabet.append(str(i))
	Allwords = [''.join(i) for i in product(alphabet, repeat = P-1)]
	#Defining allowed words of length p-1
	AllowedWords= list(Allwords)
	C=q**(P-1)
	G=[0 for i in range(C)]
	for i in range(len(Allwords)):
	init=0
	for b in forbiddenwords:
	if b in Allwords[i]:
	init+=1
	if init!=0:
	G[i]=1
	for i in range(C):
	a= Allwords[i]
	if G[i] != 0:
	AllowedWords.remove(a)
	#Creating the adjacency matrix of the subshift
	c= len(AllowedWords)
	A= [[0 for i in range(c)] for j in range(c)]
	for i in range(c):
	for j in range(c):
	g= AllowedWords[i]
	h= AllowedWords[j]
	b=1
	for k in range(P-2):
	if g[k+1] != h[k]:
	b=0
	if b==1:
	y= g + h[P-2]
	for a in forbiddenwords:
	if a in y:
	b=0
	A[i][j]=b
	#Defining the function to determine the correlation polynomials
	def getCorrelation(w1, w2):
	correlation= []
	if len(w1)== 1 and len(w2)==1:
	if w1[0]==w2[0]:
	correlation.append(1)
	if w1[0] != w2[0]:
	correlation.append(0)
	else:
	for i in range(len(w1)):
	a= min(len(w1)-i, len(w2))
	b=1
	for j in range(a):
	if w2[j] != w1[i+j]:
	b=0
	correlation.append(b)
	return correlation
	#Defining the determinant of a matrix function
	def getcofactor(m, i, j):
	return [row[: j] + row[j+1:] for row in (m[: i] + m[i+1:])]
	def determinantOfMatrix(mat):
	if(len(mat)==1):
	return mat[0][0]
	if(len(mat) == 2):
	b1 = np.polymul(mat[0][0] , mat[1][1])
	b2= np.polymul( mat[1][0] , mat[0][1])
	return np.polyadd(b1, -b2)
	Sum = 0
	for current_column in range(len(mat)):
	sign = (-1) ** (current_column)
	sub_det = determinantOfMatrix(getcofactor(mat, 0, current_column))
	Sum = np.polyadd( Sum, (sign * mat[0][current_column] * sub_det))
	return Sum
	#Defining the adjoint of a matrix function
	def adjointofMatrix(mat):
	if len(mat)==1:
	return [[1]]
	sign=1
	adj= [ [0 for i in range(len(mat))] for j in range(len(mat))]
	for i in range(len(mat)):
	for j in range(len(mat)):
	sign= (-1)**(i+j)
	temp1= getcofactor(mat, i, j)
	temp2= determinantOfMatrix(temp1)
	adj[j][i] = sign*temp2
	return adj
	#Getting the sum of matrix entries
	def getMatrixsum(mat):
	Sum=0
	for i in range(len(mat)):
	for j in range(len(mat)):
	Sum= np.polyadd(Sum, mat[i][j])
	return Sum
	#The correlation matrix M(z)
	M= [[0 for i in range(s)] for j in range(s)]
	for i in range(s):
	for j in range(s):
	temp= np.poly1d(getCorrelation(forbiddenwords[i],
	forbiddenwords[j]))
	M[j][i]= temp
	#The Perron root using theorem 1
	p1= np.polyadd(np.polymul(np.poly1d([1,-q]),
	determinantOfMatrix(M)),getMatrixsum(adjointofMatrix(M)))
	p2=determinantOfMatrix(M)
	P1=np.roots(p1).tolist()
	P2=np.roots(p2).tolist()
	## step a: removing the common factors of numerator and denominator of
	z-q+r(z)
	for a in P1:
	if a in P2:
	P1.remove(a)
	P2.remove(a)
	eigenvaluesmoduli=[]
	for a in range(len(P1)):
	S= abs(P1[a])
	eigenvaluesmoduli.append(S)
	theta= max(eigenvaluesmoduli)
	def Remove(test_list, item, number):
	## step b:removing the item for all its occurrences
	for i in range(number):
	test_list.remove(item)
	return test_list
	def Polynomialwithgivenroots(list):
	prod=np.poly1d([1])
	for a in list:
	prod= np.polymul(prod, np.poly1d([1,-a]) )
	return prod
	#Column sum of adjoint matrix
	C=[0 for i in range(s)]
	for i in range(s):
	Sum=np.poly1d([0])
	temp= list(determinantOfMatrix(M))
	B= np.poly1d(temp)
	for j in range(s):
	Sum = np.polyadd(Sum, adjointofMatrix(M)[j][i])
	if theta in np.roots(determinantOfMatrix(M)):
	k= np.roots(determinantOfMatrix(M)).tolist().count(theta)
	B=
	Polynomialwithgivenroots(Remove(np.roots(determinantOfMatrix(M)).tolist(),
	theta, k))
	Sum=Polynomialwithgivenroots(Remove(np.roots(Sum).tolist(), theta,
	k))
	C[i]= (np.polyval(Sum, theta))/(np.polyval(B, theta))
	#Row sum of adjoint matrix
	R=[0 for i in range(s)]
	for i in range(s):
	Sum=0
	temp= list(determinantOfMatrix(M))
	B= np.poly1d(temp)
	for j in range(s):
	Sum = np.polyadd(Sum, adjointofMatrix(M)[i][j])
	if theta in np.roots(determinantOfMatrix(M)):
	k= np.roots(determinantOfMatrix(M)).tolist().count(theta)
	B=
	Polynomialwithgivenroots(Remove(np.roots(determinantOfMatrix(M)).tolist(),
	theta, k))
	Sum=Polynomialwithgivenroots(Remove(np.roots(Sum).tolist(), theta,
	k))
	R[i]= np.polyval(Sum, theta)/(B(theta))
	#The eigenvectors using theorem 2.
	u= [0 for i in range(c)]
	v=[0 for i in range(c)]
	for x in range(c):
	Sum=0
	for i in range(s):
	G= forbiddenwords[i][1:]
	H= np.polyval(getCorrelation(G, AllowedWords[x]), theta)
	Sum= Sum+ R[i]*H
	u[x]= 1- Sum
	for x in range(c):
	Sum=0
	for i in range(s):
	G= forbiddenwords[i]
	H= np.polyval(getCorrelation(AllowedWords[x], G), theta)
	Sum= Sum+ C[i]*H
	v[x]= 1- Sum
	#Printing outputs
	print('The adjacency matrix A=',A)
	print('The correlation matrix, M(z)=',M)
	print('Determinant of M(z), D(z)=')
	print(determinantOfMatrix(M))
	print('Sum of entries of adjoint of M(z), S(z)=')
	print(getMatrixsum(adjointofMatrix(M)))
	print(‘The Perron eigenvalue of A is:’,theta)
	print(‘The left Perron eigenvector of A is:’, u)
	print(‘The right Perron eigenvector of A is:’,v)
\end{verbatim}
\vspace{.1in}
\noindent\textbf{Outputs for the examples}:\\
\noindent\textbf{Example 9.1}:
\begin{verbatim}
	Enter the size of symbol set:3
	Enter the number of forbidden words:1
	Enter forbidden word no.0:01
	The adjacency matrix A= [[1, 0, 1], [1, 1, 1], [1, 1, 1]]
	The correlation matrix, M(z)= [[poly1d([1, 0])]]
	Determinant of M(z), D(z)=1 x
	Sum of entries of adjoint of M(z), S(z)=[1]
	The Perron eigenvalue of A is: 2.618033988749895
	The left Perron eigenvector of A is: [1.0, 0.6180339887498949, 1.0]
	The right Perron eigenvector of A is: [0.6180339887498949, 1.0, 1.0]
\end{verbatim}

\noindent\textbf{Example 9.2}:
\begin{verbatim}
	Enter the size of symbol set:3
	Enter the number of forbidden words:1
	Enter forbidden word no.0:00
	The adjacency matrix A= [[0, 1, 1], [1, 1, 1], [1, 1, 1]]
	The correlation matrix, M(z)= [[poly1d([1, 1])]]
	Determinant of M(z), D(z)=1 x + 1
	Sum of entries of adjoint of M(z), S(z)=[1]
	The Perron eigenvalue of A is: 2.732050807568877
	The left Perron eigenvector of A is: [0.7320508075688773, 1.0, 1.0]
	The right Perron eigenvector of A is: [0.7320508075688773, 1.0, 1.0]
\end{verbatim}

\noindent\textbf{Example 9.3}:
\begin{verbatim}
	Enter the size of symbol set:5
	Enter the number of forbidden words:2
	Enter forbidden word no.0:00
	Enter forbidden word no.1:1010
	The adjacency matrix A= output omitted due to the size of the matrix
	The correlation matrix, M(z)= [[poly1d([1, 1]), poly1d([1])], [poly1d([0]),
	poly1d([1, 0, 1, 0])]]
	Determinant of M(z), D(z)=
	4 3 2
	1 x + 1 x + 1 x + 1 x
	Sum of entries of adjoint of M(z), S(z)=
	3
	1 x + 2 x
	The Perron eigenvalue of A is: 4.821125912405385
	The left Perron eigenvector of A is: [0.6222612838798931,
	0.8211259124053896, 0.8211259124053896, 0.8211259124053896,
	0.8211259124053896, 0.8211259124053896, 0.8211259124053896,
	0.8211259124053896, 0.8211259124053896, 0.8211259124053896,
	0.8211259124053896, 0.8211259124053896, 0.8211259124053896,
	0.8211259124053896, 0.8211259124053896, 0.8211259124053896,
	0.8211259124053896, 0.8211259124053896, 0.8211259124053896,
	0.8211259124053896, 0.9587514136451422, 0.9587514136451422,
	0.9587514136451422, 0.9587514136451422, 1.0, 1.0, 1.0, 1.0, 1.0, 1.0, 1.0,
	1.0, 1.0, 1.0, 1.0, 1.0, 1.0, 1.0, 1.0, 1.0, 1.0, 1.0, 1.0, 1.0, 1.0, 1.0,
	1.0, 1.0, 1.0, 1.0, 1.0, 1.0, 1.0, 1.0, 1.0, 1.0, 1.0, 1.0, 1.0, 1.0, 1.0,
	1.0, 1.0, 1.0, 1.0, 1.0, 1.0, 1.0, 1.0, 1.0, 1.0, 1.0, 1.0, 1.0, 1.0, 1.0,
	1.0, 1.0, 1.0, 1.0, 1.0, 1.0, 1.0, 1.0, 1.0, 1.0, 1.0, 1.0, 1.0, 1.0, 1.0,
	1.0, 1.0, 1.0, 1.0, 1.0, 1.0, 1.0, 1.0, 1.0, 1.0, 1.0, 1.0, 1.0, 1.0, 1.0,
	1.0, 1.0, 1.0, 1.0, 1.0, 1.0, 1.0, 1.0, 1.0, 1.0]
	The right Perron eigenvector of A is: [0.7940493563331796,
	0.992913984858676, 1.0, 1.0, 1.0, 0.8282119275467135, 0.992913984858676,
	1.0, 1.0, 1.0, 0.8282119275467135, 0.992913984858676, 1.0, 1.0, 1.0,
	0.8282119275467135, 0.992913984858676, 1.0, 1.0, 1.0, 0.8282119275467135,
	1.0, 1.0, 1.0, 0.7940493563331796, 0.992913984858676, 1.0, 1.0, 1.0,
	0.8282119275467135, 0.992913984858676, 1.0, 1.0, 1.0, 0.8282119275467135,
	0.992913984858676, 1.0, 1.0, 1.0, 0.8282119275467135, 0.992913984858676,
	1.0, 1.0, 1.0, 0.992913984858676, 1.0, 1.0, 1.0, 0.7940493563331796,
	0.992913984858676, 1.0, 1.0, 1.0, 0.8282119275467135, 0.992913984858676,
	1.0, 1.0, 1.0, 0.8282119275467135, 0.992913984858676, 1.0, 1.0, 1.0,
	0.8282119275467135, 0.992913984858676, 1.0, 1.0, 1.0, 0.992913984858676,
	1.0, 1.0, 1.0, 0.7940493563331796, 0.992913984858676, 1.0, 1.0, 1.0,
	0.8282119275467135, 0.992913984858676, 1.0, 1.0, 1.0, 0.8282119275467135,
	0.992913984858676, 1.0, 1.0, 1.0, 0.8282119275467135, 0.992913984858676,
	1.0, 1.0, 1.0, 0.992913984858676, 1.0, 1.0, 1.0, 0.7940493563331796,
	0.992913984858676, 1.0, 1.0, 1.0, 0.8282119275467135, 0.992913984858676,
	1.0, 1.0, 1.0, 0.8282119275467135, 0.992913984858676, 1.0, 1.0, 1.0,
	0.8282119275467135, 0.992913984858676, 1.0, 1.0, 1.0]
\end{verbatim}

\noindent\textbf{Example 9.4}:
\begin{verbatim}
	Enter the size of symbol set:5
	Enter the number of forbidden words:2
	Enter forbidden word no.0:0000
	Enter forbidden word no.1:0001
	The adjacency matrix A= output omitted due to the size of the matrix
	The correlation matrix, M(z)= [[poly1d([1, 1, 1, 1]), poly1d([0])],
	[poly1d([1, 1, 1]), poly1d([1, 0, 0, 0])]]
	Determinant of M(z), D(z)=
	6 5 4 3
	1 x + 1 x + 1 x + 1 x
	Sum of entries of adjoint of M(z), S(z)=
	3
	2 x
	The Perron eigenvalue of A is: 4.987087795283145
	The left Perron eigenvector of A is: [0.800776738163447, 0.800776738163447,
	0.9613467483647578, 0.9613467483647578, 0.9613467483647578,
	0.9613467483647578, 0.9613467483647578, 0.9613467483647578,
	0.9613467483647578, 0.9613467483647578, 0.9935438976415721,
	0.9935438976415721, 0.9935438976415721, 0.9935438976415721,
	0.9935438976415721, 0.9935438976415721, 0.9935438976415721,
	0.9935438976415721, 0.9935438976415721, 0.9935438976415721,
	0.9935438976415721, 0.9935438976415721, 0.9935438976415721,
	0.9935438976415721, 0.9935438976415721, 0.9935438976415721,
	0.9935438976415721, 0.9935438976415721, 0.9935438976415721,
	0.9935438976415721, 0.9935438976415721, 0.9935438976415721,
	0.9935438976415721, 0.9935438976415721, 0.9935438976415721,
	0.9935438976415721, 0.9935438976415721, 0.9935438976415721,
	0.9935438976415721, 0.9935438976415721, 0.9935438976415721,
	0.9935438976415721, 0.9935438976415721, 0.9935438976415721,
	0.9935438976415721, 0.9935438976415721, 0.9935438976415721,
	0.9935438976415721, 0.9935438976415721, 0.9935438976415721, 1.0, 1.0, 1.0,
	1.0, 1.0, 1.0, 1.0, 1.0, 1.0, 1.0, 1.0, 1.0, 1.0, 1.0, 1.0, 1.0, 1.0, 1.0,
	1.0, 1.0, 1.0, 1.0, 1.0, 1.0, 1.0, 1.0, 1.0, 1.0, 1.0, 1.0, 1.0, 1.0, 1.0,
	1.0, 1.0, 1.0, 1.0, 1.0, 1.0, 1.0, 1.0, 1.0, 1.0, 1.0, 1.0, 1.0, 1.0, 1.0,
	1.0, 1.0, 1.0, 1.0, 1.0, 1.0, 1.0, 1.0, 1.0, 1.0, 1.0, 1.0, 1.0, 1.0, 1.0,
	1.0, 1.0, 1.0, 1.0, 1.0, 1.0, 1.0, 1.0, 1.0, 1.0, 1.0, 1.0]
	The right Perron eigenvector of A is: [0.6015534763268939, 1.0, 1.0, 1.0,
	1.0, 0.9870877952831443, 1.0, 1.0, 1.0, 1.0, 0.9870877952831443, 1.0, 1.0,
	1.0, 1.0, 0.9870877952831443, 1.0, 1.0, 1.0, 1.0, 0.9870877952831443, 1.0,
	1.0, 1.0, 1.0, 0.9226934967295154, 1.0, 1.0, 1.0, 1.0, 0.9870877952831443,
	1.0, 1.0, 1.0, 1.0, 0.9870877952831443, 1.0, 1.0, 1.0, 1.0,
	0.9870877952831443, 1.0, 1.0, 1.0, 1.0, 0.9870877952831443, 1.0, 1.0, 1.0,
	1.0, 0.9226934967295154, 1.0, 1.0, 1.0, 1.0, 0.9870877952831443, 1.0, 1.0,
	1.0, 1.0, 0.9870877952831443, 1.0, 1.0, 1.0, 1.0, 0.9870877952831443, 1.0,
	1.0, 1.0, 1.0, 0.9870877952831443, 1.0, 1.0, 1.0, 1.0, 0.9226934967295154,
	1.0, 1.0, 1.0, 1.0, 0.9870877952831443, 1.0, 1.0, 1.0, 1.0,
	0.9870877952831443, 1.0, 1.0, 1.0, 1.0, 0.9870877952831443, 1.0, 1.0, 1.0,
	1.0, 0.9870877952831443, 1.0, 1.0, 1.0, 1.0, 0.9226934967295154, 1.0, 1.0,
	1.0, 1.0, 0.9870877952831443, 1.0, 1.0, 1.0, 1.0, 0.9870877952831443, 1.0,
	1.0, 1.0, 1.0, 0.9870877952831443, 1.0, 1.0, 1.0, 1.0, 0.9870877952831443,
	1.0, 1.0, 1.0, 1.0]
\end{verbatim}

\end{document}